\documentclass{article}

\usepackage{amsmath,amsfonts,amssymb}
\usepackage[toc]{appendix}
\usepackage{algpseudocode}
\usepackage{array}
\usepackage{bbm,bm}
\usepackage[usenames,dvipsnames]{xcolor}
\usepackage{comment}
\usepackage{enumerate}
\usepackage{epic}
\usepackage{footnote}
\usepackage{graphics}
\usepackage{graphicx}
\usepackage{caption}
\usepackage{subcaption}
\usepackage{import}
\usepackage[utf8]{inputenc}
\usepackage{marginnote}
\usepackage{mathtools}
\usepackage{multicol}  
\usepackage{multirow}
\usepackage{psfrag}
\usepackage{setspace}
\usepackage{tabularx}
\usepackage{url}
\usepackage[nice]{nicefrac}
\usepackage{algorithm}
\usepackage{algpseudocode}
\usepackage{booktabs}
\usepackage{microtype}
\usepackage{hyperref}
\usepackage{pgfplots}

\usepackage{clay_style}

\newtheorem{example}[thm]{Example}

\newcommand{\eto}{\stackrel{e}{\longrightarrow}}
\newcommand{\email}[1]{\protect\href{mailto:#1}{#1}}

\usepackage{definitions}

\DeclareMathOperator*{\argmin}{argmin}

\title{\sffamily Gaussian smoothing gradient descent for minimizing functions (GSmoothGD)}

\begin{document}

\author{Andrew Starnes
    \thanks{Behavioral Research Learning Lab, Lirio AI Research, Lirio, Knoxville, TN (\email{astarnes@lirio.com}, \email{adereventsov@lirio.com}, \email{cwebster@lirio.com}).}
\and Anton Dereventsov\footnotemark[1]
\and Clayton Webster\footnotemark[1]
}

\maketitle

\begin{abstract}
This work analyzes the convergence of a class of smoothing-based gradient descent methods when applied to optimization problems.
In particular, Gaussian smoothing is employed to define a nonlocal gradient that reduces high-frequency noise, small variations, and rapid fluctuations in the computation of the descent directions while preserving the structure and features of the loss landscape.
The resulting Gaussian smoothing gradient descent (GSmoothGD) approach can facilitate gradient descent in navigating away from and avoiding local minima with increased ease, thereby substantially enhancing its overall performance even when applied to non-convex optimization problems.
This work also provides rigorous theoretical error estimates on the rate of convergence of GSmoothGD iterates. These estimates exemplify the impact of underlying function convexity, smoothness, input dimension, and the Gaussian smoothing radius.
To combat the curse of dimensionality, we numerically approximate the GSmoothGD nonlocal gradient using Monte Carlo (MC) sampling and provide a theory in which the iterates converge regardless of the function smoothness and dimension.
Finally, we present several strategies to update the smoothing parameter aimed at diminishing the impact of local minima, thereby rendering the attainment of global minima more achievable.
Computational evidence complements the present theory and shows the effectiveness of the MC-GSmoothGD method compared to other smoothing-based algorithms, momentum-based approaches, and classical gradient-based algorithms from numerical optimization.


\end{abstract}



\section{Introduction}
\label{sec:intro}
Our aim is to solve for the global extrema of a high-dimensional non-convex objective function $f: \R^d \rightarrow \R$.  To achieve this goal we consider the unconstrained optimization problem, parameterized by a $d$-dimensional vector $\bmx = (x_1, \ldots, x_d) \in \R^d$, i.e.,
\begin{equation}
\label{eq:opt}
    \min_{\bmx \in \R^d} f(\bmx). 
\end{equation}
This challenging setting appears in a myriad of applications including: 
training neural networks and large-scale language models \cite{10.1093/nsr/nwad124,sun-etal-2022-bbtv2};
object detection and feature learning (e.g., \cite{zhao2019object} and the references therein);
portfolio optimization and option pricing \cite{app10020437,https://doi.org/10.1002/asmb.2209}; 
as well as material, drug, and structural design \cite{LIU2017159,sym13111976}.
  
Standard gradient-based optimization (GD) methods prove inefficient due to the abundance of local minima that can ensnare the optimization process.
Thus \eqref{eq:opt} is typically solved with a derivative- or gradient-free optimization approach~\cite{FGKM18,NesterovSpokoiny15,RiosSahinidis13,Larson_et_al_19,MWDS18,salimans2017evolution,CRSTW18,Hansen_Ostermeier_CMA_01,hansen2006cma,osher2019,dereventsov2022adaptive,10.1016/j.matdes.2020,zhang2020novel}. 
As such, throughout this effort we also assume that $f(\bmx)$ is only available by virtue of function evaluations, and the gradient $\nabla f(\bmx)$ is generally inaccessible, either because it is undefined, rapidly fluctuating, noisy, or too difficult to compute.

Of particular interest to this effort is a class of such techniques known as Gaussian smoothing (GS), 
which involves a convolution operation that integrates the product of a given target function and a Gaussian kernel over all possible shifts. 
See Figure~\ref{fig:picture_of_smoothing} for an example of how this convolution acts on $f$.
Using the Gaussian distribution, sets this apart from work done on convolutions with compactly supported distributions (e.g., see \cite[Example 7.19]{rockafellar2009variational}).
The amount of smoothing is controlled by the standard deviation of the Gaussian distribution, with larger values resulting in broader and more pronounced smoothing effects, while smaller values preserve more details of the function.  In particular, similar to the works~\cite{Nesterov_book_2004,NesterovSpokoiny15,dereventsov2022adaptive}, we introduce the notion of GS of the objective function $f(\bmx)$ in~\eqref{eq:opt}.
Let $\sigma > 0$ be a global smoothing parameter and denote by $f_\sigma(\bmx)$ the Gaussian smoothing of $f$ with radius $\sigma$, that is
\begin{equation}
\label{eq:f_sigma}
    f_\sigma(\bmx) = \frac{1}{\pi^{\nicefrac{d}{2}}} \int_{\mathbb{R}^d} f(\bmx + \sigma\bmu) 
    \, e^{-\|\bmu\|_2^2} \,\mathrm{d} \bmu
    = \mathbb{E}_{\bmu \sim \mathcal{N}(0,\mathbb{I}_d)}
    \big[ f(\bmx + \sigma\sqrt{2}\bmu) \big],
\end{equation}
where $\mathcal{N}(0,\mathbb{I}_d)$ is a standard $d$-dimensional Gaussian distribution.
We remark that $f_\sigma$ preserves important features and structure of the objective function including, such as convexity, $L$ smoothness, and is often differentiable even when $f$ is not for $\sigma>0$.

\begin{figure}
    \centering
    \def\s{1}
    \begin{tikzpicture}[scale=0.6]
        \begin{axis}[
            axis lines = center,
            xtick={0},
            ytick={0},
            xmin=-2,
            xmax=2,
            ymin=-1.5,
            ymax=2,
        ]
        \addplot [domain=-2:2, samples=100,latex-latex,dashed,thick]{x^4-2*x^2+cos(deg(2*3.14*x))} node[right,pos=0.75] {$f$};
        \addplot [domain=-2:2, samples=100,latex-latex,thick]{x^4+(3*\s^2-2)*x^2+0.75*\s^4-\s^2+cos(deg(2*3.14*x))*2.72^(-9.87*\s^2)} node[right,pos=0.551] {$f_{\s}$};
        \end{axis}
    \end{tikzpicture}
    \def\s{0.5}
    \begin{tikzpicture}[scale=0.6]
        \begin{axis}[
            axis lines = center,
            xtick={0},
            ytick={0},
            xmin=-2,
            xmax=2,
            ymin=-1.5,
            ymax=2,
        ]
        \addplot [domain=-2:2, samples=100,latex-latex,dashed,thick]{x^4-2*x^2+cos(deg(2*3.14*x))} node[right,pos=0.75] {$f$};
        \addplot [domain=-2:2, samples=100,latex-latex,thick]{x^4+(3*\s^2-2)*x^2+0.75*\s^4-\s^2+cos(deg(2*3.14*x))*2.72^(-9.87*\s^2)} node[left,pos=0.625] {$f_{\s}$};
        \end{axis}
    \end{tikzpicture}
    \def\s{0.1}
    \begin{tikzpicture}[scale=0.6]
        \begin{axis}[
            axis lines = center,
            xtick={0},
            ytick={0},
            xmin=-2,
            xmax=2,
            ymin=-1.5,
            ymax=2,
        ]
        \addplot [domain=-2:2, samples=100,latex-latex,dashed,thick]{x^4-2*x^2+cos(deg(2*3.14*x))} node[right,pos=0.75] {$f$};
        \addplot [domain=-2:2, samples=100,latex-latex,thick]{x^4+(3*\s^2-2)*x^2+0.75*\s^4-\s^2+cos(deg(2*3.14*x))*2.72^(-9.87*\s^2)} node[left,pos=0.739] {$f_{\s}$};
        \end{axis}
    \end{tikzpicture}
    \caption{Example of how Gaussian smoothing transforms $f(x)=x^4-2x^2+\cos(2\pi x)$. For $\sigma\geq\sqrt{\nicefrac{2}{3}}$, $f_{\sigma}$ is convex despite the fact that $f$ is not convex. Minimizing $f_1$ is much easier than minimizing $f$. As $\sigma\to 0$, $f_{\sigma}$ approaches $f$.}
    \label{fig:picture_of_smoothing}
\end{figure}
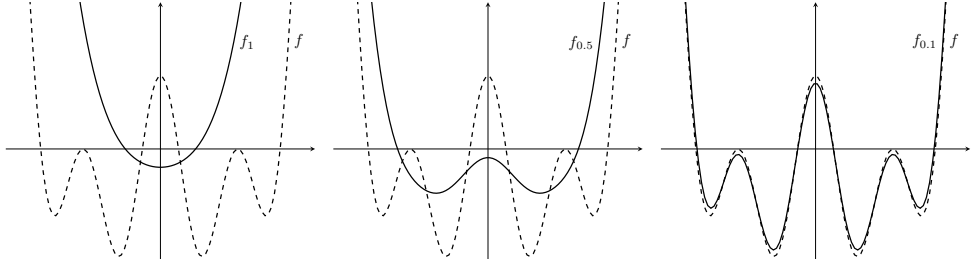

For well-behaved functions (e.g., $f$ is Lipschitz or its gradient is), we can exploit that the gradient of $f_\sigma(\bmx)$ is given by 
\begin{equation}
\label{eq:fs_grad}
	\nabla f_\sigma(\bmx) 
	= \frac{2}{\sigma \pi^{\nicefrac{d}{2}}}
	    \int_{\mathbb{R}^d} \bmu f(\bmx + \sigma \bmu) \, e^{-\|\bmu\|_2^2} \, \mathrm{d}\bmu
    = \frac{2\sqrt{2}}{\sigma}\, \mathbb{E}_{\bmu \sim \mathcal{N}(0,\mathbb{I}_d)}
        \big[ \bmu f(\bmx + \sigma\sqrt{2} \bmu) \big],
\end{equation}
to replace the standard local gradient with a nonlocal approximation, which aims to reduce high-frequency noise, small variations, and rapid fluctuations in the computation of the descent directions.  The resulting Gaussian smoothing gradient descent (GSmoothGD), described by Algorithm~\ref{alg:GSmoothGD}, aims to solve \eqref{eq:opt} by replacing it with the following $(d+1)$-dimensional surrogate problem

\begin{equation}
\label{eq:surrogate_opt}
    \min_{(\sigma,\bmx) \in \mathbb{R}^+\times\R^d} f_{\sigma}(\bmx),
\end{equation}
and satisfies
\begin{equation}
    \inf_{\bmx\in\mathbb{R}^d}f_{\sigma}(\bmx)
    \to\inf_{\bmx\in\mathbb{R}^d}f(\bmx)
\end{equation}
as $\sigma\to 0$.
%
%
%
The GSmoothGD approach can facilitate gradient descent in navigating away from and avoiding local minima with increased ease, thereby substantially enhancing the overall performance when applied to non-convex optimization problems.
As such, this work also provides rigorous theoretical error estimates on the GSmoothGD iterates rate of convergence, that exemplify the impact of underlying function convexity, smoothness, and input dimension, as well as the Gaussian smoothing radius.  As far as we are aware, this is the first detailed analysis of convergence rates of GSmoothGD for general non-convex, $L$-smooth problems \eqref{eq:opt}.
See the end of the Section~\ref{sec:related_work} for a more detailed discussion of this.
In addition, to provide a better theoretical understanding of GSmoothGD we also rigorously estimate the relationship between the minimizers of $f$ and $f_{\sigma}$ for convex, $L$-smooth functions.

\begin{algorithm}
    \caption{Gaussian smoothing gradient descent (GSmoothGD)}
    \label{alg:GSmoothGD}
    \begin{algorithmic}[1]
        \Require $f:\mathbb{R}^d\to\mathbb{R}$, $(\sigma_k)_{k=1}^{K}$, $x_0\in\mathbb{R}^d$, $t>0$
        \For{$k=1\to K$}
            \State $x_k = x_{k-1} - t\nabla f_{\sigma_k}(x_{k-1})$
        \EndFor 
    \end{algorithmic}
\end{algorithm}

Numerical estimation of the search directions is an additional problem due to the complex nature of objective functions in high-dimensional spaces. 
A popular set of approaches to compute the gradient of a Gaussian smoothed function are known as evolutionary strategies (ES)~\cite{salimans2017evolution,Hansen_CMA,hansen2006cma,liu2019trust,CRSTW18}, which are a class of algorithms inspired by natural evolution~\cite{Wierstra_NES_14}.    
Recently, interest in ES methods has been reinvigorated and has become a popular approach in several machine learning problems such as training neural networks~\cite{Such-GA17,MorseStanley16,Cui_NeurIPS18} and reinforcement learning~\cite{zhang2020accelerating,Khadka-CERL,SigaudStulp13,salimans2017evolution}.
Covariance matrix adaptation evolution strategy (CMA-ES) is a popular evolutionary algorithm used for continuous optimization problems. Several important papers have contributed to the understanding of the convergence properties of CMA-ES \cite{hansen2006cma,Hansen_CMA,Hansen_Ostermeier_CMA_01}.  These works explore the convergence theory of CMA-ES under various settings, including scenarios with limited function evaluations, large population sizes, and even in the presence of noise, which is often encountered in real-world optimization problems. 
As such, these methods are particularly useful when solving optimization problems related to nonconvex and nonsmooth objective functions.
We note that the current work is concerned with the convergence properties of the smoothing-based optimization algorithms. Therefore, our numerical experiments exclude the ES and CMA-ES algorithms because they operate differently from Algorithm~\ref{alg:GSmoothGD}, which is the main focus of our study.
More recently, another class of ES strategies has emerged, known as directional Gaussian smoothing (DGS)~\cite{zhang2020accelerating,tran2023convergence,zhang2020novel,dereventsov2022adaptive,10.1016/j.matdes.2020}, 
that specifically averages the original objective function along $d$ orthogonal directions.  As such, the representation of the partial derivatives of the smoothed function along these directions can be reduced to  one-dimensional integrals, as opposed to the more complex $d$-dimensional integrals used in standard ES methods. Consequently, one can approximate the averaged partial derivatives using deterministic \cite{zhang2020accelerating,tran2023convergence,zhang2020novel} or adaptive \cite{dereventsov2022adaptive} Gauss-Hermite quadrature rules.   

However, for increased accuracy, both ES and DGS approaches must expand the size of the domain, and expend more effort approximating either the $d$-dimensional integral or the integral within each individual direction. The resulting explosion in computational effort is a symptom of the {\em curse of dimensionality}.  In order to combat this grand challenge, similar to \cite{iwakiri2022single,mobahi2015link}, we numerically approximate the $d$-dimensional GSmoothGD nonlocal gradient \eqref{eq:fs_grad} with the use of Monte Carlo (MC) sampling (e.g., see \cite{Fishman_96,Gunzburger:2014hi,Caflisch} and the references therein), called MC-GSmoothGD, and provide a theory in which the iterates converge regardless of the function smoothness and dimension.
To the best of our knowledge, this constitutes the inaugural in-depth analysis of MC-GSmoothGD convergence rates with a varying smoothing parameter in the context of high-dimensional non-convex problems. Moreover, this work also addresses the calculation of the smoothing parameter $\sigma$, with the details described in Section \ref{sec:compute_sigma}. 
Again, see the end of the Section~\ref{sec:related_work} for a more detailed discussion of this.

\subsection{Paper organization}
\label{sec:toc}
%
The paper is organized as follows.
The remainder of Section \ref{sec:intro} is focused on describing an important connection between Gaussian smoothing (or homotopy methods) and a class of partial differential equations (PDEs) known as the heat equation \cite{evans2010partial}, as well as providing an overview of previous works.
Section~\ref{sec:notation} introduces the notation and nomenclature used throughout.
In Sections \ref{sec:non-convex}, we first describe the simplicity of GSmoothGD and then provide the first complete convergence analysis of such approaches for both high-dimensional convex and non-convex functions.
In Section \ref{sec:mins}, we provide a thorough analysis of the  relationship between minimizers and smoothed minimizers. 
In Section \ref{sec:num_GSmoothGD}, we describe various numerical approximation schemes for Gaussian smoothing, including MC, variance of the gradient of the smoothed function, and Gaussian homotopy.
We also discuss the various strategies employed to appropriately update the smoothing parameter $\sigma$.
In Section \ref{sec:numerics}, we numerically compare Gaussian smoothing with other smoothing-based algorithms, as well as other conventional methods such as momentum-based approaches, and classical gradient-based algorithms from numerical optimization.
In Section \ref{sec:conclusions}, we provide several concluding remarks and possible future directions we are excited to pursue.
Finally, Appendix \ref{app:notation} contains the proofs of the background results from Section~\ref{sec:notation} and Appendix \ref{app:convergence} contains the proofs of the main convergence results from Sections~\ref{sec:convex}, \ref{sec:non-convex}, \ref{sec:MCGSmoothGD}.

\subsection{Related works}
\label{sec:related_work}

Gaussian smoothing has been studied across many disciplines, which creates two almost distinct paths of research.
We divide the work that has been done so far into these two areas of research, the first being from the perspective of evolutionary strategies and the second being motivated by homotopy continuation and partial differential equations (PDEs).

\subsection*{Connections to evolutionary strategies}
\label{sec:evol_strat}

Motivated by~\cite{dereventsov2022adaptive},
this effort can appropriately be classified in a category of methods known as smoothing-based and even gradient-free optimization.  The development and analysis of such techniques are extensive, see, e.g.,~\cite{mobahi2012seeing,mobahi2012gaussian,mobahi2016closed,hazan2016graduated,mobahi2015theoretical,mobahi2015link,mobahi2016training,jin2018local,duchi2015optimal,zhang2017hitting,iwakiri2022single,Nesterov_book_2004,NesterovSpokoiny15} and the references therein. However, to date, this is the first detailed analysis of convergence rates of GSmoothGD for general non-convex problems, given by \eqref{eq:opt}, has been unavailable in the current and existing literature. 
Several works have enriched our comprehension of the convergence characteristics of CMA-ES, as evidenced by studies such as \cite{hansen2006cma,Hansen_CMA,Hansen_Ostermeier_CMA_01,salimans2017evolution,liu2019trust,CRSTW18,Cui_NeurIPS18}. These investigations delve into the convergence theory of CMA-ES across a spectrum of scenarios, encompassing limited function evaluations, large population sizes, and the incorporation of noise, a common real-world optimization challenge. As a result, these methodologies offer significant utility in the optimization of non-convex and non-smooth objective functions.

The recent effort \cite{tran2023convergence} investigates the convergence of 
directional Gaussian smoothing (DGS), first introduced in~\cite{zhang2020accelerating,zhang2020novel},
in the very specific case when the objective function consists of a convex component perturbed by oscillating noise. The theory demonstrates that the DGS iterations exhibit exponential convergence toward a narrowed region around the solution for strongly convex functions. The extent of this region is defined by the wavelength of the noise.  Furthermore, they establish a connection between the optimal values of the Gaussian smoothing radius and the noise wavelength. This connection validates the benefits of employing a moderate or large smoothing radius in this particular situation.

\subsection*{Connections to Gaussian homotopy methods and PDEs}
\label{sec:pdes}

Recall that the heat equation is the initial value problem given by
\begin{equation}
\label{eqn:heateqn}
    \left\{\begin{array}{ll}
        u_t=\epsilon\Delta u\qquad&(\bmx,t)\in\mathbb{R}^d\times(0,\infty)\\
        u(\bmx,0)=f(\bmx)&\bmx\in\mathbb{R}^d
    \end{array}\right.
\end{equation}
The fundamental solution to the heat equation (HE-PDE) is given by
\begin{equation}
    u(\bmx,t)=\mathbb{E}_{\bmu \sim \mathcal{N}(0,2\epsilon^2t\mathbb{I}_d)}[f(\bmu)].
\end{equation}
This means that the smoothing function $f_{2\sqrt{t}}(\bmx)$ is a solution to the heat equation as well.
Explicitly, we have
\begin{equation}
\label{eqn:smoothing_heateqn}
    \frac{\partial}{\partial\sigma}f_{\sigma}(\bmx)
    =\frac{\sigma}{2}\Delta_{\bmx} f_{\sigma}(\bmx).
\end{equation}
In other words, using the heat equation smoothes the initial value function.

The heat equation is a particular example of the viscous Hamilton-Jacobi PDE which is given by
\begin{equation}
    \left\{\begin{array}{ll}
        u_t=\epsilon\Delta u-H(\nabla u)\qquad&(\bmx,t)\in\mathbb{R}^d\times(0,\infty)\\
        u(\bmx,0)=f(\bmx)&\bmx\in\mathbb{R}^d
    \end{array}\right.
\end{equation}
In particular, when $H(\cdot)=0$ we recover the heat equation.
In the case that $\epsilon=0$, we have the standard Hamilton-Jacobi PDE (HJ-PDE), which can be used to smooth functions as well.
HJ-smoothing attempts to remove narrow minima from $f$ regardless of depth, whereas HE-smoothing attempts to remove shallow minima regardless of width.

Motivated by~\cite{sagun2016eigenvalues} and statistical physics, the effort~\cite{baldassi2016local} introduces optimizing via local entropy loss, which favors minima with low eigenvalues in Hessian (and hence wider minima).  The authors also introduce Entropy-SGD and compare numerical results on image classification problems (using RNNs and CNNs) with heat equation smoothing, concluding this approach is 
{\em very different from local entropy and can introduce an artificial minimum between two nearby sharp valleys which is detrimental to generalization}.
However, these artificial minima do not persist when $\sigma\to 0$ (i.e., when finding the minimum of $f$).
For analysis of how the HE-smoothed minimizers relate to the original minimizers, see Section~\ref{sec:mins}.
Furthermore, the work~\cite{chaudhari2018deep} shows that the local entropy loss is a solution to the viscous HJ-PDE, and that viscous HJ-PDE smoothing has much better empirical performance as well as improved (theoretical) convergence rates compared to GD.  

Inspired by~\cite{chaudhari2018deep}, the effort~\cite{osher2022laplacian} introduces Laplacian smoothing (LSGD), see Algorithm~\ref{alg:lsgd}, as $x_{k+1}=x_k-t_kA^{-1}_{\sigma}\nabla f_{i_k}(x_k)$, where $A_{\sigma}$ is a particular circulant convolution matrix and $\nabla f_{i_k}$ is the SGD gradient using observation $i_k$ (that is, HJ-PDE with $H(\cdot)=\frac{1}{2}\langle\cdot,A^{-1}_{\sigma}\cdot\rangle$).
Modified Laplacian smoothing gradient descent (mLSGD) is given in \cite{kreusser2022deterministic} where the smoothing parameter is allowed to change between iterations, denoted by $\sigma(k)$.
The authors show that this approach is less susceptible to saddle points and converges as fast as (and anecdotally faster than) GD.
We remark that there is no mathematical reason that the matrix, denoted $A_{\sigma(k)}$ for mLSGD, \textit{needs} to be the graph Laplacian, in fact, any symmetric, positive definite matrix will satisfy the HJ-PDE described in \cite{osher2022laplacian}.

\begin{algorithm}
    \caption{LSGD}
    \label{alg:lsgd}
    \begin{algorithmic}[1]
        \Require $f:\mathbb{R}^d\to\mathbb{R}$, $\sigma\geq 0$, $x_0\in\mathbb{R}^d$, $t>0$
        \For{$k=1\to K$}
            \State $x_k = x_{k-1} - t A_{\sigma}^{-1}\left(\nabla f(x_{k-1})\right)$
        \EndFor 
    \end{algorithmic}
\end{algorithm}

One major benefit of using HE-smoothing is that we do not need access to the gradient of $f$ and instead can just use function evaluations to approximate the gradient of $f_{\sigma}$.
This cannot be done with HJ-smoothing, which is why $\nabla f(x)$ is required for Algorithm~\ref{alg:lsgd}.

In addition to the connection with PDEs, smoothing can be motivated from optimization via homotopy continuation.\footnote{In this context, this method could also be called Graduated Non-Convexity (see \cite{blake1987visual}).}
Optimization via homotopy continuation, parameterized by $t$,
creates a homotopy between the original $f$ ($t=0$) and a smooth function ($t=1$).
The method finds the minimizers of the homotopy at $t$, starting at $t=1$ and iteratively reduces $t$ to 0.
That is, it starts by finding the minimizers of a nice function and repeats this with less smooth functions until it arrives at the original function.
Explicitly, let $h:\mathbb{R}^d\times [0,T]\to\mathbb{R}$ where $h(x,0)=h(x)$, $h(x,T)$ is ``smooth'', and $h(x,t)$ is continuously differentiable for $x\in\mathbb{R}^d$ and $t>0$.
Denote the path of the minimizer by $x(t)$, where for $t\geq 0$, $x(t)$ satisfies $\nabla g(x(t),t)=0$ and is continuous.
Obviously, such a path may not exist.
Furthermore, there's not always a clear choice for how to construct the homotopy.
When the homotopy is created by convolving the function with a Gaussian kernel (where $\sigma$ both acts as $t$ and controls the variance), this method is called Gaussian Homotopy Continuation  (see Algorithm~\ref{alg:homotopycontinuation}).
In \cite{mobahi2015link}, they show that evolving the objective function using  a Gaussian homotopy provides optimal convexification of the function.

\begin{algorithm}
    \caption{Optimization by Gaussian Homotopy Continuation}
    \label{alg:homotopycontinuation}
    \begin{algorithmic}[1]
        \item Input: $f:\mathbb{R}^d\to\mathbb{R}$, $\{\sigma_k\}_{k=1}^{K}$ s.t. $0<\sigma_{k+1}<\sigma_k$, $x_0\in\mathbb{R}^d$
        \For{$k=1\to K$}
            \State $x_k =$ local minimizer of $f_{\sigma_k}(x)$ initialized at $x_{k-1}$
        \EndFor
        \item Output: $x_K$
    \end{algorithmic}
\end{algorithm}

In practice, this method leads to an algorithm with two loops, an outer loop that decreases $\sigma_k$ and an inner loop that performs gradient descent on $f_{\sigma_k}$.
Motivated by the connection between the derivative with respect to the smoothing parameter and the second space derivative given by the heat equation, \cite{iwakiri2022single} propose a single loop Gaussian homotopy continuation method (SLGH) (see Algorithm~\ref{alg:slgh}), which is a particular instance of GSmoothGD.
In~\cite{iwakiri2022single}, they prove that SLGH converges for functions that are Lipschitz and have Lipschitz gradients.

\begin{algorithm}
    \caption{Deterministic Single Loop GH algorithm}
    \label{alg:slgh}
    \begin{algorithmic}[1]
        \Require Iteration number $T$, initial solution $x_1$, initial smoothing parameter $\sigma_1$, step size $\beta$ for $x$, step size $\eta$ for $t$, decreasing factor $\gamma\in(0,1)$, sufficient small positive value $\epsilon$
        \For{$k=1\to T$}
            \State $x_{k+1}=x_k-\beta\nabla f_{\sigma_k}(x_k)$
            \State $\sigma_{k+1}=\left\{\begin{array}{ll}
                \gamma \sigma_k\text{ (SLGH$_r$)}\\
                \max\{\min\{\sigma_k-\eta\frac{\partial}{\partial\sigma}f_{\sigma_k}(x_k),\gamma\sigma_k\},\epsilon\}\text{ (SLGH$_d$)}
            \end{array}\right.$
        \EndFor
    \end{algorithmic}
\end{algorithm}

Since both~\cite{NesterovSpokoiny15} and~\cite{iwakiri2022single} study variants of SmoothGD, we want to explicitly point out how our analysis differs from their results.
Neither paper focuses on the minimizers of $f_{\sigma}$, which makes our analysis about their behavior entirely novel.
As it relates to this paper, the primary focus of \cite{NesterovSpokoiny15} is on the ``random gradient-free oracles'' which are Monte Carlo estimates of (\ref{eq:fs_grad}).
The majority of our results use the actual gradient of the smoothed function, but our study of MC-GSmoothGD uses two of these oracles.
The analysis done for non-convex functions in \cite{NesterovSpokoiny15} is most similar to our analysis of MC-GSmoothGD, but they fix $\sigma$ and our focus is to examine the impact that changing $\sigma$ has on convergence.
Furthermore, they mention that $\sigma$ could be allowed to change, but they omit any proof of convergence because it ``is quite long and technical''.
We prove this result in Theorem~\ref{thm:mcGSmoothGD}.

As discussed before, \cite{iwakiri2022single} is very closely related to this paper.
However, despite the similarity of some of the results, they assume that both the target function and its gradient are Lipschitz; we require either one of these, but not both and achieve similar convergence results.
This means that our analysis applies to a much larger class of functions.
Additionally, in the convex case, we improve upon their results and still only assume that either the function or its derivative is Lipschitz.
Finally, \cite{iwakiri2022single} makes the heavy restriction that $f$ is bounded (see the discussion after Lemma~\ref{lem:differentsmoothingvalues} for more details).


\section{Background and preliminaries}
\label{sec:notation}
Throughout this effort, we assume that $f:\mathbb{R}^d\rightarrow\mathbb{R}$. 
We often make the standard assumption that our functions are $L$-smooth and convex. 
We restate the definitions here for the convenience of the reader.
\begin{defn}
\label{def:fsmooth}
We say that $f$ is $M$-Lipschitz if for all $\bmx,\bmy\in\mathbb{R}^d$
\s
    \|f(\bmx)-f(\bmy)\|\leq M\|\bmx-\bmy\|.
\f
We say that $f$ is $L$-smooth if $\nabla f$ is $L$-Lipschitz.
\end{defn}

Instead of using the definition of $L$-smoothness, the majority of the time, we use the following well-known equivalent formulation (which we only state as an implication of the definition).

\begin{prop}
\label{prop:altdefflsmooth}
If $f:\mathbb{R}^d\to\mathbb{R}$ is $L$-smooth, then for $\bmx,\bmy\in\mathbb{R}^d$
\start
\label{eqn:proplsmooth}
    \big| f(\bmy)-f(\bmx)-\langle\nabla f(\bmx),\bmy-\bmx\rangle \big| \leq\frac{L}{2}\|\bmy-\bmx\|^2.
\finish
\end{prop}

\begin{defn}
We say that $f:\mathbb{R}^d\to\mathbb{R}$ is convex if for all $\bmx,\bmy\in\mathbb{R}^d$ with $\bmx\not=\bmy$ and for all $t\in(0,1)$ we have
\s
    f \big( t\bmx+(1-t)\bmy \big) \leq tf(\bmx)+(1-t)f(\bmy).
\f
If this inequality is strict, then we say $f$ is strictly convex.
\end{defn}

Since we typically assume that $f$ is $L$-smooth, it is differentiable. When $f$ is differentiable, the following equivalent definition is used.

\begin{prop}
\label{prop:altdeffconvex}
If $f:\mathbb{R}^d\to\mathbb{R}$ is differentiable and convex, then for all $\bmx,\bmy\in\mathbb{R}^d$
\s
    f(\bmy)-f(\bmx)\geq\langle\nabla f(\bmx),\bmy-\bmx\rangle.
\f
\end{prop}

The remainder of this section will outline some new theoretical estimates that are required to prove our main Theorems, given by Theorem \ref{thm:GSmoothGD_convex} and \ref{thm:GSmoothGD_non-convex}. Detailed proofs of theoretical results described thorough this section are given in Appendix \ref{app:notation}. 
The first result shows that smoothing essentially maintains some of the key properties of the function.
First, smoothing $f$ maintains its convexity and Lipschitz constant (these are found in~\cite{NesterovSpokoiny15})  and preserves its smoothness in part.
Second, smoothing a Lipschitz function induces the derivative of the smoothed function to be Lipschitz as well
(these two results can also be found in~\cite{NesterovSpokoiny15}).
%

\begin{lem}\label{lem:fsigmalsmoothandconvex}
\cite{NesterovSpokoiny15} Let $f:\mathbb{R}^d\to\mathbb{R}$ and $\sigma> 0$.
\begin{enumerate}[(a)]
    \item  If $f$ is $L$-smooth then $f_{\sigma}$ is also $L$-smooth.
    \item  If $f$ is convex then so is $f_{\sigma}$.
    \item If $f$ is $M$-Lipschitz then $f_{\sigma}$ is $M$-Lipschitz and $\frac{M\sqrt{2d}}{\sigma}$-smooth.
\end{enumerate}
\end{lem}

\begin{rem}
As discussed in the introduction, in practice, we do not need to have access to the gradient of $f$ and instead rely on function evaluations to compute the derivative of $f_{\sigma}$.
In fact, even for the theory to hold we do not always have to assume that the target function $f$ is differentiable everywhere because we apply results like Proposition~\ref{prop:altdefflsmooth} and~\ref{prop:altdeffconvex} on $f_{\sigma}$ which will be differentiable even if $f$ is only differentiable a.e. (for $\sigma>0$).
For example, if $f$ is Lipschitz then it is automatically differentiable a.e. or if $f$ is only $L$-smooth a.e. then $f_{\sigma}$ is $L$-smooth for $\sigma>0$ (the set of measure 0 where $f$ is non-differentiable does not impact the proof of Lemma~\ref{lem:fsigmalsmoothandconvex}).
\end{rem}
Since we are using smoothing to optimize a function, we need to know how the values of the smoothing function relate to the values of the original function. The following result shows that smoothing the function increases the minimum values and, for convex functions, the output unless $f$ is constant.
While these are standard results for the heat equation (e.g., see Theorem 6 of Section 2.3 in~\cite{evans2010partial} or Theorem 3.1 of~\cite{iwakiri2022single}), we include a direct proof without appealing to the heat equation since (\ref{eq:fs_grad}) may not be satisfied\footnote{Although, the heat equation is satisfied for $L$-smooth or $M$-Lipschitz functions (see discussion after proof of Lemma~\ref{lem:gradient_goes_into_convolution}).}.
\begin{lem}\label{lem:fsigmagreaterthanf}
Let $f:\mathbb{R}^d\to\mathbb{R}$ and $\sigma > 0$.
\begin{enumerate}[(a)]
    \item Suppose $f_{\sigma}(\bmx)=\inf_{\bmy\in\mathbb{R}^d}f(\bmy)>-\infty$ for some $\bmx\in\mathbb{R}^d$, then $f$ is constant.
    \item If $f$ is non-constant and either $f(\bmx)\geq m$ or $f(\bmx)\leq M$, then $f_{\sigma}(\bmx)> m$ or $f_{\sigma}(\bmx)< M$ (respectively).
    \item If $f$ is convex, then $f_{\sigma}(\bmx)\geq f(\bmx)$ whenever $f$ is differentiable at $\bmx\in\mathbb{R}^d$. (\cite[Eqn. (11)]{NesterovSpokoiny15})
\end{enumerate}
\end{lem}
Intuitively, smoothing an already smoothed function is the same as smoothing the original function by some amount; the following result explicitly shows this amount.
\begin{lem}\label{lem:chainofsmoothing}
Let $f:\mathbb{R}^d\to\mathbb{R}$ and $\sigma,\tau>0$. Define 
$    
    \eta = \sqrt{\sigma^2+\tau^2},
$
then
$
    (f_{\sigma})_{\tau}(\bmx) = f_{\eta}(\bmx)
$
where $(f_{\sigma})_{\tau}$ is the smoothing function of $f_{\sigma}$ with smoothing parameter $\tau$.
\end{lem}

\begin{rem}
Assuming that $f_{\sigma}$ satisfies the heat equation~\ref{eqn:heateqn}\footnote{For example, this would occur if $f$ is $L$-smooth (see Remark~\ref{rem:l-smooth_functions_satisfy_heat_eqn}) or decays fast enough (see~\cite[Chapter 8 Exercise 44]{folland2009fourier}).}, then instead of the self-contained proof in the appendix, we could just appeal to the ``heat semigroup'' as follows:
Define
\begin{equation}
    (P_tf)(\bmx)=u(\bmx,t),
\end{equation}
where $u$ is the solution to the heat equation with initial value $f(\bmx)$, then
\begin{equation}
    P_{t_1}\circ P_{t_2} = P_{t_1+t_2}.
\end{equation}
Then in our case (as $\sigma = 2\sqrt{t}$), for some $\sigma_1,\sigma_2,\sigma\geq 0$, we have
\begin{equation}
    \frac{\sigma^2}{4}=t=t_1+t_2=\frac{\sigma_1^2}{4}+\frac{\sigma_2^2}{4}
    \Longrightarrow
    \sigma = \sqrt{\sigma_1^2+\sigma_2^2}.
\end{equation}
\end{rem}

Our results rely on careful analysis of how much the smoothing changes the output. We saw above how this relationship unfolds at the minimum, but we need to more exactly how much the smoothed function can differ from the original function.

\begin{lem}
\label{lem:differentsmoothingvalues}
Let $\tau\geq\sigma\geq 0$.
\begin{enumerate}[(a)]
    \item If $f$ is $L$-smooth, then
        $|f_{\tau}(\bmx)-f_{\sigma}(\bmx)|\leq(\tau^2-\sigma^2)\nicefrac{Ld}{4}$.
    \item If $f$ is $M$-Lipschitz, then
        $|f_{\tau}(\bmx)-f_{\sigma}(\bmx)|\leq M|\tau-\sigma|\sqrt{\nicefrac{d}{2}}$.
\end{enumerate}
\end{lem}
If $f$ is convex and we do not specify that $\sigma\leq\tau$, the first part of the previous lemma becomes
\begin{equation}
    f_{\tau}(\bmx)-f_{\sigma}(\bmx)\leq\frac{Ld}{4}\max(0,\tau^2-\sigma^2),
\end{equation}
which is often the form in which we use the result.
Also, note that this bound is optimal; if $f(\bmx)=\|\bmx\|^2$, then $f_{\sigma}(\bmx)=f(\bmx)+\frac{\sigma^2Ld}{4}$ which shows that the bound cannot be improved in general.

While this result appears innane, it turns out to have some important implications.
First, since the bound is independent of $\bmx$, this shows that $f_{\sigma}\to f$ uniformly as $\sigma\to 0$.
Second, it also means that $f_{\sigma}(\bmx)$ exists (and is finite) for all $\bmx$ and $\sigma>0$.
Next, as mentioned before, this result shows that~\cite{iwakiri2022single} requires $f$ be bounded.
Assumption A1 (i) in~\cite{iwakiri2022single} states that for all $\sigma>0$
\begin{equation}
    \sup_{x\in\mathbb{R}^d}E_u(|f(x+\sigma u)|)<\infty.
\end{equation}
This means that
\begin{align}
    \sup_{x\in\mathbb{R}^d} |f_{\sigma}(x)|
    &=\sup_{x\in\mathbb{R}^d} |E_u(f(x+\sigma u))|
    \leq \sup_{x\in\mathbb{R}^d} E_u(|f(x+\sigma u)|)
    <\infty.
\end{align}
Applying the previous result for any fixed $\sigma>0$, we have
\begin{align}
    \sup_{x\in\mathbb{R}^d}|f(x)|
    &\leq \sup_{x\in\mathbb{R}^d} |f_{\sigma}(x)|+\frac{Ld\sigma^2}{4}
    <\infty.
\end{align}
So, $f$ is bounded, meaning that even well-behaved functions like $\|x\|^2$ are outside of the scope of the analysis of~\cite{iwakiri2022single}.

%
One last result relating the smoothed function's output to the original's output is needed.
So far, we know that for a particular $\bmx$-value, the change in output values is bounded by a function of $\sigma$.
The same is true when comparing the minimum of the smoothed function to the minimum of the original function.
This also means if we fix $\sigma$ and minimize $f_{\sigma}$, then we know we are within some distance of the minimum of the original function and that distance gets smaller as $\sigma$ does.
\begin{cor}\label{cor:fminusfsigma}
Let $f:\mathbb{R}^d\to\mathbb{R}$ achieve its minimum and $\sigma\geq 0$. Let $\xstar$ and $\sstar$ be any minimizers of $f$ and $f_{\sigma}$, respectively.
\begin{enumerate}[(a)]
    \item If $f$ is $L$-smooth, then $0\leq f_{\sigma}(\sstar)-f(\xstar)\leq\frac{1}{4}\sigma^2Ld$.
    \item If $f$ is $M$-Lipschitz, then 
        $0\leq f_{\sigma}(\sstar)-f(\xstar)\leq M\sigma\sqrt{\nicefrac{d}{2}}$.
\end{enumerate}
\end{cor}

Intuitively, one reason we want to smooth this way is to make the function ``more convex''.
Unbounded, non-convex functions that have function values go to $\infty$ and $-\infty$ in certain directions, e.g., $f(x)=x^3$, can only become convex if flattened out entirely, which will not happen for any finite smoothing value.
Even for bounded, non-convex functions, the following corollary shows that it is not always possible to make the function convex.
In particular, if you have a bounded, non-convex function, then no amount of smoothing makes the function convex.
This result is a complement to the work of Mobahi (e.g., see \cite{mobahi2015theoretical}).

\begin{cor}\label{cor:corollarytofsigma0}
Let $f:\mathbb{R}^d\to\mathbb{R}$ so that $|f_{\sigma}(\bmx)|\leq Ae^{a\|x\|^2}$ for all $0\leq\sigma\leq S$ for some constants $A$ and $a$ and $f_{\sigma}$ satisfies the heat equation (\ref{eqn:heateqn}).
\begin{enumerate}[(a)]
    \item If $f_{\sigma}$ is constant for some $\sigma\leq S$, then $f_{\eta}=f$ for all $\eta>0$.
    \item If $g:\mathbb{R}^d\to\mathbb{R}$ satisfies $|g_{\sigma}(\bmx)|\leq Be^{b\|x\|^2}$ for all $0\leq\sigma\leq S$ for some constants $B$ and $b$ and $f_{\sigma}(\bmx) = g_{\sigma}(\bmx)$ for all $x\in\mathbb{R}^d$, then $f=g$ a.e.
    \item If there exists $S\geq\sigma>0$ such that $f_{\sigma}(\bmx)=f_{\tau}(\bmx)$ for all $\tau\geq\sigma$, then $f=f_{\eta}$ a.e. for all $\eta>0$.
    \item If $f$ is non-convex and bounded, then $f_{\sigma}$ is non-convex for all $\sigma\geq 0$.
\end{enumerate}
\end{cor}

Note that (a) and (b) follow from the uniqueness of solutions to the heat equation (e.g., see Section 2.3 of \cite{evans2010partial}).

\begin{proof}[Proof of (c)]
Let $\eta>0$ and set $\tau = \sqrt{\sigma^2+\eta^2}$. Then since $\tau>\sigma$, using Lemma~\ref{lem:fsigmagreaterthanf},
\begin{align}
    (f_{\eta})_{\sigma}=f_{\tau}=f_{\sigma}.
\end{align}
Then by part (b), we have that $f_{\eta} = f$ a.e.
\end{proof}

\begin{proof}[Proof of (d)]
Since $f$ is bounded, $f_{\sigma}$ is bounded by Lemma~\ref{lem:fsigmagreaterthanf}.
If $f_{\sigma}$ were convex, then it would be constant.
By part (a), we would have that $f_{\sigma}=f=0$.
However, $f$ is non-convex, which means it cannot be constant.
As such, $f_{\sigma}$ must be non-convex.
\end{proof}

The last result in this section rigorously shows that, for the type of functions we consider in this paper. we can represent $\nabla f_{\sigma}$ as in (\ref{eq:fs_grad}). The proof uses the dominated convergence theorem, so we need the following result.

\begin{lem}
\label{lem:finl1}
Suppose $f_{\sigma}(\bmx)\in\mathbb{R}$ for all $\bmx\in\mathbb{R}^d$ and suppose that $f$ is bounded below. Then $f\in L^1(\mathbb{R}^d,k_{\sigma})$ where $k_{\sigma}(\bmx)=\frac{1}{\sigma^d}e^{-\frac{\|\bmx\|^2}{\sigma^2}}$.
\end{lem}

If $f$ is bounded below and is either $L$-smooth or $M$-Lipschitz, then by Lemma~\ref{lem:differentsmoothingvalues}, we know that $f_{\sigma}(\bmx)\in\mathbb{R}$ for any $\bmx\in\mathbb{R}^d$. This means that $f\in L^1(\mathbb{R}^d,k_{\sigma})$. With this result in hand, we can show that the partial derivatives can pass into either argument of the convolution.

\begin{lem}
\label{lem:gradient_goes_into_convolution}
Let $f:\mathbb{R}^d\to\mathbb{R}$ with $f\in L^1(\mathbb{R}^d,k_{\sigma})$
Then $\nabla f_{\sigma}$ exists and for $i\in\{1,...,d\}$
\begin{equation}
    \frac{\partial f_{\sigma}}{\partial x_i}(\bmx)
    =\left(\frac{\partial f}{\partial x_i}\star k_{\sigma}\right)(\bmx)
    =\left(f\star \frac{\partial k_{\sigma}}{\partial x_i}\right)(\bmx).
\end{equation}
\end{lem}

Note that the proof of this lemma can be adapted to show that if $f$ is bounded below and either $L$-smooth or $M$-Lipschitz, then $f_{\sigma}$ satisfies the heat equation.

\section{Gaussian smoothing gradient descent (GSmoothGD)}
\label{sec:GSmoothGD}

Our work is primarily concerned with a complete convergence 
analysis of Gaussian smoothing gradient descent (GSmoothGD) approaches
governed by the following update rule:
\begin{equation}
\label{eq:GSmoothGD}
    \bmx_{k+1}=\bmx_{k} - t\,\nabla f_{\sigma_{k+1}}(\bmx_{k}).
\end{equation}
Here the standard gradient $\nabla f(\bmx_k)$ is replaced with the Gaussian smoothed surrogate $\nabla f_{\sigma_{k+1}}(\bmx_{k})$, given by \eqref{eq:fs_grad},
with $(0,\bmx_k)$ and $(0,\bmx_{k+1})$ being possible solutions to \eqref{eq:surrogate_opt} at iterations $k$ and $k+1$, 
$t$ is the step size or learning rate, and $\sigma_{k}$ is the smoothing radius utilized during the $k$th step of \eqref{eq:GSmoothGD}.


\subsection{Convergence of GSmoothGD for high-dimensional convex functions}
\label{sec:convex}
We prove the following convergence estimate of GSmoothGD when applied to 
high-dimensional convex functions.  The detailed proof can be found in Section \ref{app:GSmoothGD_convex}.
\begin{thm}[Convergence of GSmoothGD for convex functions]
\label{thm:GSmoothGD_convex}
Suppose $f:\mathbb{R}^d\to\mathbb{R}$ is convex, $L$-smooth, and has at least one minimizer.
Let $(\sigma_n)_{n=1}^{\infty}$ be a sequence of positive real numbers. 
Then, after $k$ iterations of GSmoothGD, defined by \eqref{eq:GSmoothGD}, with a fixed step size $0< t\leq\frac{1}{L}$, the solution $f(\bmx_k)$ satisfies
\begin{equation}
\label{eq:f1}
    f(\bmx_{k})-f(\xstar)\leq
    \frac{\|\bmx_0-\xstar\|^2}{2tk}+\frac{Ld}{4k}\left(
    \sum_{i=1}^{k}\sigma_{i}^2
    +\sum_{i=2}^{k}i\max(0,\sigma_{i}^2-\sigma_{i-1}^2)
    \right),
\end{equation}
where $\xstar$ is any minimizer of $f$.
\end{thm}

\begin{rem}[Constant smoothing radius]
\label{rem:GSmoothGD_convex}
If we further assume that $\sigma_n=\sigma >0$ for all $n$,
then for $t\leq \frac{1}{L}$, and with the use of Lemma~\ref{lem:differentsmoothingvalues} we get that
\begin{align*}
    \begin{split}
    0
    \leq f(\bmx_{k})-f(\xstar)
    \leq\frac{\|\bmx_0-\xstar\|^2}{2tk}+\frac{Ld\sigma^2}{4}.
    \end{split}
\end{align*}
Additionally, the same proof can be used to show a similar result for convex, $M$-Lipschitz functions. In this case $L$-smooth is replaced with $\frac{M\sqrt{2d}}{\sigma}$-smooth and $\frac{\sigma^2Ld}{4}$ is replaced with $M\sigma\sqrt{\frac{d}{2}}$.
\end{rem}

The proof of Theorem~\ref{thm:GSmoothGD_convex} is a modification of the proof that the standard GD algorithm converges when used to minimize convex functions (see Section 1.2.3 of \cite{Nesterov_book_2004} for example).  In order to generalize the theory of GD to include the use of the gradient of the Gaussian smoothing of a convex function we exploit several new results shown below.  %
First, we show that the Gaussian smoothing functions inherit convexity and $L$-smoothness from $f$ (Lemma~\ref{lem:fsigmalsmoothandconvex}). Second, we detail how the function values of $f$ and $f_{\sigma}$ are related (Lemmas~\ref{lem:fsigmagreaterthanf} and~\ref{lem:differentsmoothingvalues}). Third, we show that the smaller the smoothing value, the closer the Gaussian smoothing function is to the original function (at least when the original function is convex) (Lemma~\ref{lem:differentsmoothingvalues}).


\subsection{Convergence of GSmoothGD for high-dimensional non-convex functions}
\label{sec:non-convex}
In the setting of non-convex functions when a minimizer exists, we no longer have that increasing the smoothing parameter also increases the output of the smoothing function. Hence, changing $\sigma_i$ between iterations increases the bound on the gradient. This is the main result of this section, which is stated in the following theorem.

\begin{thm}[Convergence of GSmoothGD for non-convex functions]\label{thm:GSmoothGD_non-convex}
Suppose $f:\mathbb{R}^d\to\mathbb{R}$ is non-convex, $L$-smooth and let $(\sigma_n)_{n=1}^{\infty}$ be a sequence of positive real numbers. 
Then, after $k$ iterations of GSmoothGD, defined by \eqref{eq:GSmoothGD}, with a fixed step size $0< t\leq\frac{1}{L}$,
\begin{multline}
    \min_{i=1,...,k}\|\nabla f(\bmx_i)\|^2
    \leq
        \frac{4}{tk}\big(f_{\sigma_1}(\bmx_1)-f_{\sigma_{k+1}}(\bmx_{k+1})\big)\\
        +\frac{Ld}{2tk}\sum_{i=1}^{k}|\sigma_{i+1}^2-\sigma_i^2|
        +\frac{L^2(6+d)^3}{4k}\sum_{i=1}^{k}\sigma_{i+1}^2
\end{multline}
\end{thm}
%
%

\begin{prop}\label{prop:complexity}
Let $f$ be $L$-smooth and $(\sigma_{n})_{n=1}^{\infty}$ be square summable. Then the complexity of GSmoothGD is $\mathcal{O}\left(\frac{1+d^3}{\epsilon^2}\right)$.
\end{prop}
Note that the increase in complexity compared to~\cite{NesterovSpokoiny15} comes from the fact that we vary the smoothing parameter (rather than forcing it to be sufficiently small) and use $\nabla f_{\sigma_k}$ rather than a finite difference approximation.

\begin{rem}
\label{rem:constant}
If $\sigma_n=\sigma$ for all $n$ (similarly, if $\sum_{i=0}^{\infty}\sigma_i^2$ diverges), then we can't compute the iterations in the same manner. 
In the constant case, we end up with:
\begin{equation}
    \min_{i=1,...,k}\|\nabla f(\bmx_i)\|^2
    \leq\frac{4}{tk}\big(f(\bmx_0)-f(\xstar)\big)
    +\frac{L^2(6+d)^3\sigma^2}{4}.
\end{equation}
So if we want the derivative to be small, we have to make $\sigma$ small too (but we could do this by taking a finite number of unique $\sigma_n$).
\end{rem}

\subsection{Relationship between minimizers and smoothed minimizers}
\label{sec:mins}

In this section, we discuss the relationship between any minimizers of $f$, $\xstar$, and $f_{\sigma}$, $\sstar$, for convex, $L$-smooth functions.
Our motivation comes from the following question: if we can optimize $f$ (i.e. $\xstar$ exists), then can we optimize $f_{\sigma}$ for a fixed $\sigma$ (i.e. does $\sstar$ exist)?
The first result, which is just the contrapositive of Corollary~\ref{cor:fminusfsigma}, says that if $f(\bmx)$ is too large then $\bmx$ cannot minimize $f_{\sigma}$ (Corollary~\ref{cor:contrapositive}).
The second result shows that if the set of minimizers is bounded, then for small enough values of $\sigma$ we know that $\sstar$ exists (Proposition~\ref{prop:fbddsmoothedminsexist}).
This proof requires a technical bound on the smoothing functions (Lemma~\ref{lem:fsigmaxlessthanfsigmay}).
Third, we show that if the set of minimizers of $f$ is bounded then for small enough $\sigma$ that the smooth minimizer $\sstar$ is close to some actual minimum $\xstar$ (Proposition~\ref{prop:smoothedminimizersclosetooriginals}).
We conclude this section with a discussion of epigraphical convergence.
%
Before we show these results, we provide an example to show that if the set of minimizers of $f$ is unbounded, then no minimizers of $f_{\sigma}$ for any $\sigma>0$ may exist.
\begin{example}\label{example:unboundedminimizers}
Let $f:\mathbb{R}\to\mathbb{R}$ be defined as
\begin{equation}
    f(x)=\left\{\begin{array}{ll}
        -x\qquad&x<0\\
        0&x\geq 0
    \end{array}\right..
\end{equation}
Then for all $x\in\mathbb{R}$
\begin{equation}
    \nabla f_{\sigma}(x)
    =\frac{1}{\sqrt{\pi}}\int_{-\infty}^{-\nicefrac{x}{\sigma}}-e^{-u^2}\;du<0.
\end{equation}
This means that $f_{\sigma}$ has no minimum regardless of $\sigma$.
So even though $f(x)$ is minimized by any positive number, $f_{\sigma}(x)$ has no minimum for any $\sigma > 0$.
\end{example}
For the remainder of this section, for any $\epsilon\geq 0$, we will denote the inverse image of $[f(\xstar),f(\xstar)+\epsilon]$ as $f^{\star}(\epsilon)$, defined by
\begin{equation}
    f^{\star}(\epsilon)=\{\bmx\in\mathbb{R}^d:f(\bmx)-f(\xstar)\leq\epsilon\}.
\end{equation}
The following first result in this section allows us to limit the size of the set of possible minimizers of $f_{\sigma}$ solely based on the output of the original $f$.
\begin{cor}\label{cor:contrapositive}
Let $f:\mathbb{R}^d\to\mathbb{R}$ be $L$-smooth and $\sigma\geq 0$.  If
\begin{equation}
    |f(\bmx)-\inf f|>\frac{\sigma^2Ld}{4},
\end{equation}
then $\bmx$ is not a minimizer of $f_{\sigma}$.
In particular, any minimizer of $f_{\sigma}$ is in $f^{\star}(\nicefrac{\sigma^2Ld}{4})$.
\end{cor}

We now focus on showing our second result with the technical lemma required in the proof of Proposition~\ref{prop:fbddsmoothedminsexist}.
\begin{lem}\label{lem:fsigmaxlessthanfsigmay}
Let $f:\mathbb{R}^d\to\mathbb{R}$ be $L$-smooth and assume that a minimizer $\xstar$ of $f$ exists.
Let $\sigma > 0$ and let $\epsilon\geq\frac{4\sigma^2Ld}{3}$ then for $\bmx\in f^{\star}\left(\frac{\sigma^2Ld}{4}\right)$
\begin{equation}
    f_{\sigma}(\bmx) < f_{\sigma}(\bmy)\qquad\forall \bmy\notin f^{\star}(\epsilon).
\end{equation}
\end{lem}
\begin{proof}
Let $\bmx\in f^{\star}\left(\frac{\sigma^2Ld}{4}\right)$ and $\bmy\notin f^{\star}(\epsilon)$ (that is, $f(\bmy) > f(\xstar) + \epsilon$).
Using Lemma~\ref{lem:differentsmoothingvalues} to switch between $f_{\sigma}$ and $f$, we have
\begin{align}
\begin{split}
    f_{\sigma}(\bmx)
    &\leq f(\bmx) + \frac{\sigma^2Ld}{4}
     \,\leq \, f(\xstar) + \frac{\sigma^2Ld}{2}
    =f(\xstar)+\frac{3\sigma^2Ld}{4}-\frac{\sigma^2Ld}{4}
    \\
    &\ \ \ < f(\xstar)+\epsilon-\frac{\sigma^2Ld}{4}
    <f(\bmy)-\frac{\sigma^2Ld}{4}
    \leq f_{\sigma}(\bmy)
\end{split}
\end{align}
\end{proof}
\begin{rem}
\label{rem:fcon1}
If $f$ is convex, the same result holds using $\epsilon\geq\frac{\sigma^2Ld}{2}$ as in this case $f(\bmy)\leq f_{\sigma}(\bmy)$.
\end{rem}
%
We can now prove our second result that provides a sufficient condition on when a minimizer of $f_{\sigma}$ can exist.
Note that this results does not assume that a minimizer of $f$ even exists.
\begin{prop}\label{prop:fbddsmoothedminsexist}
Let $f:\mathbb{R}^d\to\mathbb{R}$ be $L$-smooth.
Suppose $f^{\star}(\epsilon)$ is bounded for some $\epsilon > 0$.
Let $\sigma > 0$ with $\sigma < \sqrt{\frac{4\epsilon}{3Ld}}$.
Then $\sstar$ exists.
\end{prop}
\begin{proof}
Since $f$ is continuous, $f^{\star}(\epsilon)$ is also closed.
By assumption, $f^{\star}(\epsilon)$ is bounded, which means it is compact too.
This guarantees that $f_{\sigma}|_{f^{*}(\epsilon)}$ attains its minimum on $f^{*}(\epsilon)$.
Let
\begin{equation}
    \bmx_{\sigma}^{\epsilon} \in \argmin_{\bmx\in f^{*}(\epsilon)}f_{\sigma}(\bmx).
\end{equation}
Since $\frac{\sigma^2Ld}{4} < \epsilon$, $f^{\star}(\nicefrac{\sigma^2Ld}{4})\subseteq f^{\star}(\epsilon)$ and so
$f_{\sigma}(\bmx_{\sigma}^{\epsilon})\leq f_{\sigma}(\bmx)\;\; \forall \bmx\in f^{\star}\left(\nicefrac{\sigma^2Ld}{4}\right)$.
By Lemma \ref{lem:fsigmaxlessthanfsigmay},
$\displaystyle
    f_{\sigma}(\bmx_{\sigma}^{\epsilon})\leq f_{\sigma}(\bmy)\;\;\forall \bmy\notin f^{\star}(\epsilon), 
$
which means that
$\displaystyle
    f_{\sigma}(\bmx_{\sigma}^{\epsilon})\leq f_{\sigma}(\bmx)\;\; \forall \bmx\in \mathbb{R}^d,
$
making $\bmx_{\sigma}^{\epsilon}$ a minimizer of $f_{\sigma}$.
\end{proof}
\begin{rem}
\label{rem:fcon2}
If $f$ is convex, then we have the same result for $\sigma < \sqrt{\frac{2\epsilon}{Ld}}$.
\end{rem}
%

If $f$ is convex, then the set of minimizers being bounded means that the set of points that output just a little more than the minimum is also bounded.
On the other hand, the following proposition does not necessarily hold for non-convex functions. In fact for non-convex functions, it can be the case that $f^{\star}(0)$ is bounded and for all $\epsilon>0$ $f^{\star}(\epsilon)$ is unbounded. This can occur when there is a sequence $(\bmx_n)_{n=1}^{\infty}$ where $\|\bmx_n\|\to\infty$ and $f(\bmx_n)\to f(\xstar)$.

\begin{prop}\label{prop:convexfunctionshaveboundedalmostminimizersets}
Let $f:\mathbb{R}^d\to\mathbb{R}$ be convex and $L$-smooth with a bounded, nonempty minimizer set, $f^{\star}(0)$.
Then there exists $\epsilon>0$ so that $f^{\star}(\epsilon)$ is bounded too.
\end{prop}

\begin{proof}
Let $f:\mathbb{R}^d\to\mathbb{R}$ be convex and $L$-smooth with a bounded minimizer set, $f^{\star}(0)$.
Suppose towards a contradiction that $f^{\star}(\epsilon)$ is unbounded for all $\epsilon>0$.
Since $f^{\star}(0)$ is bounded, there is an $M>0$ so that $f^{\star}(0)\subseteq B(0,\frac{M}{2})$.
Since $f^{\star}(\frac{1}{n})$ is unbounded for all $n\in\mathbb{N}$, there exists $\bmx_n\in f^{\star}(\frac{1}{n})\cap B(0,M)^c$.
For any $0<t<1$, using the fact that $f$ is convex and the choice of $x_n$, we have
\begin{equation}
    f(t\xstar+(1-t)\bmx_n)
    \leq tf(\xstar)+(1-t)f(\bmx_n)
    < f(\xstar)+\frac{1}{n}.
\end{equation}
For each $n\in\mathbb{N}$, there exists $t_n\in(0,1)$ so that
\begin{equation}
    \bmy_n=t_n\xstar+(1-t_n)\bmx_n
    \text{ and }
    \|\bmy_n\|=M.
\end{equation}
This means $f(\bmy_n)\to f(\xstar)$.
Since $\{\bmy_n\}_{n=1}^{\infty}$ is bounded, there is a convergent subsequence $(\bmy_{n_i})_{i=1}^{\infty}$ to the point $\bmy_0$.
Since $\|\bmy_n\|=M$, $\|\bmy_0\|=M$ too.
This would mean that $\bmy_0\in f^{\star}(0)$, but since $f^{\star}(0)\subseteq B(0,\frac{M}{2})$, this cannot be the case and we arrive at our contradiction.
Therefore, for some $\epsilon>0$, $f^{\star}(\epsilon)$ is bounded.
\end{proof}

We can now prove our third result, which shows that the set of minimizers of $f_{\sigma}$ grow in a continuous manner when $f$ is convex.
\begin{prop}\label{prop:smoothedminimizersclosetooriginals}
Let $f:\mathbb{R}^d\to\mathbb{R}$ be $L$-smooth and $(\sigma_n)_{n=1}^{\infty}$ with $\sigma_n\geq 0$ and $\sigma_n\to 0$.
Suppose $f^{\star}(\delta)$ is bounded for some $\delta > 0$.
Then for all $\epsilon>0$, there exists $N\in\mathbb{N}$ such that for all $n\geq N$
\begin{equation}
    d(\xstar_{\sigma_n},f^{\star}(0))
    =\inf\left\{\|\xstar_{\sigma_n}-\xstar\|:\xstar\in f^{\star}(0)\right\}
    <\epsilon
\end{equation}
where $\xstar_{\sigma_n}$ is any minimizer of $f_{\sigma_n}$.
\end{prop}

\begin{proof}
Let $f:\mathbb{R}^d\to\mathbb{R}$ be $L$-smooth and $(\sigma_n)_{n=1}^{\infty}$ with $\sigma_n\geq 0$ and $\sigma_n\to 0$.
Since $\sigma_n\to 0$, there exists $N\in\mathbb{N}$ so that $\sigma_n\leq\sqrt{\frac{4\delta}{3Ld}}$.
By Proposition~\ref{prop:fbddsmoothedminsexist}, $\xstar_{\sigma_n}$ exists for all $n\geq N$.
Suppose, towards a contradiction, that there is some $\epsilon>0$ so that for all $i\in\mathbb{N}_{\geq N}$ there exists $n_i\geq i$ with $d(\xstar_{\sigma_{n_i}},f^{\star}(0))\geq\epsilon$.
This means that no subsequence of $(\xstar_{\sigma_{n_i}})_{i=N}^{\infty}$ can converge to a point in $f^{\star}(0)$.
Since $\sigma_n\to 0$, we still have that $\sigma_{n_i}\to 0$.
By Corollary~\ref{cor:fminusfsigma}
\begin{equation}
    |f(\xstar_{\sigma_{n_i}})-f(\xstar)|
    \leq\frac{\sigma_{n_i}^2Ld}{4}
    \leq\delta,
\end{equation}
so by Corollary~\ref{cor:contrapositive} $\xstar_{\sigma_{n_i}}\in f^{\star}(\delta)$ for all $i\geq N$.
Since $f^{\star}(\delta)$ is bounded, there exists a subsequence $(\xstar_{\sigma_{n_{i_j}}})_{j=1}^{\infty}$ that converges to a point, say $\bmx_0$.
Since
\begin{equation}
    |f(\xstar_{\sigma_{n_{i_j}}})-f(\xstar)|
    \leq\frac{\sigma_{n_{i_j}}^2Ld}{4}
    \to 0
\end{equation}
(again using Corollary~\ref{cor:fminusfsigma}) and $f$ is continuous, $f(\bmx_0)=f(\xstar)$.
However, this means that $\bmx_0\in f^{\star}(0)$, which is a contradiction.
\end{proof}

\begin{cor}
Let $f:\mathbb{R}^d\to\mathbb{R}$ be convex and $L$-smooth with a unique minimizer $\xstar$. Then for any $\sigma_n\to 0$ with $\sigma_n\geq 0$,
\begin{equation}
    \xstar_{\sigma_n}\to \xstar,
\end{equation}
where $\xstar_{\sigma_n}$ is any minimizer of $f_{\sigma_n}$.
\end{cor}

\begin{proof}
Let $\sigma_n\to 0$ with $\sigma_n\geq 0$ and $\epsilon>0$.
Since $f$ is convex, there exists $\delta>0$ so that $f^{\star}(\delta)$ is bounded.
This means there is an $N\in\mathbb{N}$ so that
\begin{equation}
    \|\xstar_{\sigma_n}-\xstar\|
    = d(\xstar_{\sigma_n},f^{\star}(0))
    < \epsilon
\end{equation}
for all $n\geq N$, where the first equality holds since $\xstar$ is unique.
This shows $\xstar_{\sigma_n}\to \xstar$.
\end{proof}

\begin{rem}
The results discussed so far are related to the notion of epigraphical convergence.

Let $g,g_n:\mathbb{R}^n\to\mathbb{R}$ for $n\in\mathbb{N}$. We say that $g_n$ epigraphically converges or epi-converges to $g$, denoted $g_n\eto g$, if 
\begin{equation}
    \liminf_{n\to\infty} f_n(x_n)\geq f(x)\text{ for any sequence }x_n\to x
\end{equation}
and
\begin{equation}
    \limsup_{n\to\infty} f_n(x_n)\leq f(x)\text{ for at least one sequence }x_n\to x.
\end{equation}

The motivation for epi-convergence is the type of convergence needed to make sure that limits of both minimums and minimizers behave as expected.
In particular, if $g_n\eto g$, then $\limsup\inf g_n\leq\inf f$~\cite[Proposition 7.30]{rockafellar2009variational} and if $x_n$ minimizes $g_n$ and $x_n\to x$ then $g_n\eto g$ (this is a crude version of~\cite[Theorem 7.11]{rockafellar2009variational}).
However, it is possible that $g_n$ has no minimizer for all $n$ despite the fact that $g_n\eto g$ and $g$ has at least one minimizer (this is exactly the case for Example~\ref{example:unboundedminimizers}).
The results that we have shown in this section provide the existence of minimizers of $f_{\sigma}$ which epi-convergence alone cannot guarantee.
Regardless, Lemma~\ref{lem:differentsmoothingvalues} gives us uniform convergence, so by~\cite[Theorem 7.15]{rockafellar2009variational} since $f$ is continuous, we do have epi-convergence of our smoothing functions.
\end{rem}

\section{Numerical methods for GSmoothGD}
\label{sec:num_GSmoothGD}
%
In this section we describe various numerical approximation schemes for Gaussian smoothing, including MC,  variance of the gradient of the smoothed function, and Gaussian homotopy.  We also discuss the strategies employed to appropriately update the smoothing parameter $\sigma$.

\subsection{Monte Carlo GSmoothGD (MC-GSmoothGD)}
\label{sec:MCGSmoothGD}
%
In practice, we cannot evaluate $\nabla f_{\sigma}(\bmx)$ exactly.
Instead, we can use one of the approximations from \cite{nesterov2017random}.
Define $\delta_{\sigma}(\bmx;\bmu)$ as either of the finite difference schemes
\begin{equation}\label{eqn:findiffschemes}
    \frac{f(\bmx+\sigma \bmu)-f(\bmx)}{\nicefrac{\sigma}{2}}
    \quad\text{or}\quad
    \frac{f(\bmx+\sigma \bmu)-f(\bmx-\sigma u)}{\sigma}
\end{equation}
and
\begin{equation}
    \bmg_{\sigma}(\bmx;N)
    =\frac{1}{N}\sum_{n=1}^{N}\delta_{\sigma}(\bmx;\bmu_n)\bmu_n
\end{equation}
where $\bmu_n$ are independent samples from the $e^{-\|\bmu\|^2}$ density.
Regardless of the choice, both are unbiased estimates of $\nabla f_{\sigma}(\bmx)$.
We will use the second choice for the MC-GSmoothGD algorithm (Algorithm~\ref{alg:mcGSmoothGD}).
Theorem~\ref{thm:mcGSmoothGD} shows the convergence of this algorithm.

\begin{algorithm}
    \caption{MC-GSmoothGD}
    \label{alg:mcGSmoothGD}
    \begin{algorithmic}[1]
        \Require $f:\mathbb{R}^d\to\mathbb{R}$, $N\in\mathbb{N}$, $\sigma>0$, $\bmx_0\in\mathbb{R}^d$, $t>0$
        \For{$k=1\to K$}
            \State Sample $\bmu_n\sim\mathcal{N}(0,\frac{1}{2})$ for $n=1,...,N$
            \State $\bmg_k=\displaystyle\frac{1}{N}\sum_{n=1}^{N}\frac{f(\bmx_{k-1}+\sigma \bmu_n)-f(\bmx_{k-1}-\sigma \bmu_n)}{\sigma}\bmu_n$
            \State $\bmx_k = \bmx_{k-1} - t\bmg_k$
        \EndFor 
    \end{algorithmic}
\end{algorithm}

\begin{lem}\label{lem:varianceofmcGSmoothGD}
Let $\delta_{\sigma}(\bmx;\bmu)$ and $\bmg_{\sigma}(\bmx;N)$ be either of the approximations of $\nabla f_{\sigma}(\bmx)$ above.
Then
\begin{equation}
    E(\|\bmg_{\sigma}(\bmx;N)\|^2)
    =\frac{1}{N}E\big(\delta_{\sigma}(\bmx;\bmu_1)^2\|\bmu_1\|^2\big)+\left(1-\frac{1}{N}\right)\|\nabla f_{\sigma}(\bmx)\|^2.
\end{equation}
\end{lem}

\begin{thm}\label{thm:mcGSmoothGD}
Let $f$ be $L$-smooth and bounded below. 
Let $(\sigma_n)_{n=1}^{\infty}$ be a sequence of positive real numbers. 
Pick $0< t<\frac{1}{2L(d+4)}$. 
Let $\delta_{\sigma}(\bmx;\bmu)$ be either of the differences in (\ref{eqn:findiffschemes}) and $\bmg_{\sigma}(\bmx;N)$ be the corresponding $N$-point Monte Carlo approximation of $\nabla f_{\sigma}(\bmx)$.
Define
\begin{equation}
    \bmx_{k+1}=\bmx_k-t\bmg_{\sigma_{k+1}}(\bmx_k;N).
\end{equation}
Then for $N\geq 1$,
\begin{multline}
    \min_{i=0,...,k-1}E(\|\nabla f(\bmx_i)\|^2)\\
    \leq\frac{1}{k}\left(\frac{N}{N-2Lt(d+4)}\right)\Bigg(\frac{4(f(\bmx_0)-f_{\sigma_{k}}(\bmx_{k}))}{t}
    +\left(\frac{L^3t(d+6)^3}{4N}+\frac{2Ld}{t}\right)\sum_{i=1}^{k}\sigma_{i}^2\Bigg).
\end{multline}
\end{thm}

If we choose
\begin{equation}
    t = \frac{1}{L(d+4)}
    \text{ and }
    \sigma_i\leq \mathcal{O}\left(\frac{\epsilon}{L^{\nicefrac{1}{2}}(d+6)^{\nicefrac{1}{2}}}\right),
\end{equation}
then, as with \cite{NesterovSpokoiny15}, the upper bound for the expected number of steps is $O(\frac{d}{\epsilon^2})$.

In terms of computational cost, the MC-GSmoothGD algorithm is more expensive than conventional gradient-based methods.
This is mainly due to the extensive computations involved in estimating the smooth gradient $g_k$ in step 3, which is computed by averaging $N$ Monte Carlo samples.
Consequently, the computational burden of the MC-GSmoothGD algorithm amounts to $\mathcal{O}(N)$ floating-point operations per iteration.
For comparison, estimating the conventional gradient vector $\nabla f(x_k)$ requires $\mathcal{O}(d)$ floating-point operations per iteration.
Furthermore, if the gradient $\nabla f$ is given analytically, this computation can be performed in $\mathcal{O}(1)$ floating-point operations.
In our numerical illustrations, we set $d = 100$ and $N = 1000$, resulting in comparable computational costs for the presented algorithms.

\subsection{Updating the smoothing parameter}
\label{sec:compute_sigma}

Regardless of how $\sigma$ is chosen or updated, the results that we have shown here apply to any of the Gaussian smoothing algorithm we have mentioned.
In the homotopy continuation setting (Algorithm~\ref{alg:homotopycontinuation}), for each $\sigma_k$, we minimize (via gradient descent) $f_{\sigma_k}$.
In the context of our results, this corresponds to creating a sequence of smoothing parameters where we repeat $\sigma_k$ as many times as the number of steps in the gradient descent.
Since the goal is to minimize $f_{\sigma_k}$, the double loop is baked into this algorithm.

Instead of focusing on minimizing each $f_{\sigma_k}$, we can perform gradient descent on the $\mathbb{R}^{d+1}$ function $f_{\sigma}(x)$.
This is what SLGH does by doing a gradient descent steps simultaneously in $\bmx$ and $\sigma$.
We know that the minimum of $f_{\sigma}(\bmx)$ over both $\bmx$ and $\sigma$ occurs when $\sigma = 0$ (see Lemma~\ref{lem:fsigmagreaterthanf}), so if SLGH finds the global minimum of $f_{\sigma}(\bmx)$ then it has found the global minimum of $f$.

Without exception, the Gaussian smoothing algorithms we have discussed (Algorithms \ref{alg:homotopycontinuation} and \ref{alg:slgh}) have all required $\sigma$ to decrease.
The motivation for this decrease is to force $\sigma\to0$ so that the original function can be minimized.
In the context of homotopy continuation, if we trace the path of minimums, then we can easily become trapped in a local minimum of $f_{\sigma}$ or $f$.
This can happen even if we start at the global minimum of $f_{\sigma}$ for some $\sigma>0$.
Note that restrictive assumptions can be made to stop this from happening, but we want to avoid further restricting $f$.
The motivation for smoothing is to use smoothing to remove local minimums, but this does not happen when $\sigma$ is small.
So if we force $\sigma$ to decrease and we become trapped in a local minimum, we are unlikely to escape unless we increase $\sigma$ high enough that our gradient descent step moves us away from this trap.

More constructively, suppose the sequence of $x_k$ basically stops changing.
Using Lemma~\ref{lem:fsigmagreaterthanf}, if $f_{\sigma_k}(x_k)<f(x_k)$, then we know the algorithm is stuck in a suboptimal stationary point. At this moment, we could entirely restart with a different initial guess or we could greatly increase $\sigma$ in hopes of smoothing out this local minimum and continue the algorithm. Rather than heurestically or randomly choosing $\sigma_k$, this approach is theoretically sound.

\section{Numerical experiments}
\label{sec:numerics}
%
In this section we compare smoothing-based optimization algorithms with other conventional methods.
We consider smoothing-based algorithms (MC-GSmoothGD, DGS~\cite{zhang2020novel}, LSGD~\cite{osher2022laplacian}, SLGH~\cite{iwakiri2022single}), backpropagation-based algorithms (NAG~\cite{sutskever2013importance}, Adam~\cite{kingma2014adam}, RMSProp~\cite{hinton2012neural}), and classical gradient-based algorithms from numerical optimization (BFGS~\cite{nocedal1999numerical}, CG (Conjugate Gradient)~\cite{nocedal1999numerical}).
The selection of backpropagation- and gradient-based algorithms is based on their widespread usage in practical applications.

Our experiments are performed in \texttt{Python3.8} with the use of \texttt{Numpy} and \texttt{Scipy} packages.
In order to provide a uniform comparison between different algorithm families, each of the smoothing-based and backpropagation-based algorithms is implemented as a custom \texttt{scipy.optimize} method with the same gradient estimation and stopping criteria.
The experiments are performed on a consumer-grade desktop and the source code is publicly available at~\url{https://github.com/sukiboo/smoothing_based_optimization}.

Our testbed consists of well-established multi-dimensional functions from the Virtual Library of Simulation Experiments\footnote{\url{https://www.sfu.ca/~ssurjano/optimization.html}}.
Specifically, we use $100$-dimensional formulations of Ackley, Levy, Michalewicz, Rastrigin, Rosenbrock, and Schwefel functions to test and compare the algorithms.
For Schwefel function we additionally take the absolute value to receive an unconstrained optimization problem.
We note that such a modification does not change the function values or the position of minimizers on the intended domain $[-500, 500]^{100}$.

Each experiment is performed $10$ times with a series of initial guesses (synchronized between different algorithms) that are randomly sampled from the corresponding domain for each test function.
For each algorithm and target function we consider the following increasingly noisy settings:
\begin{enumerate}
    \item each target function evaluation is computed exactly;
    \item each target function evaluation is computed with $0.000001$\% relative perturbation;
    \item each target function evaluation is computed with $0.01$\% relative perturbation.
\end{enumerate}

Results of our experiments are presented in Figure~\ref{fig:numerics_100d_noise}, where for each algorithm the so1lid line indicates the median value across all $10$ experiments and the shaded region denotes the area between $25$-th and $75$-th percentile.

\subsection{Hyperparameters}
For each algorithm and test function we perform a hyperparameter search for learning rate $\lambda$ on the grid [1, $10^{-1}$, $10^{-2}$, $10^{-3}$, $10^{-4}$, $10^{-5}$, $10^{-6}$].
For smoothing-based algorithms we additionally perform a hyperparameter search for the value of the smoothing parameter $\sigma$ on the grid [1, $10^{-1}$, $10^{-2}$, $10^{-3}$].

Hyperparameter searches are performed with the exact function evaluations and the final values of hyperparameters for each algorithm and test function are given in Tables~\ref{tab:hyperparameters_100d_learning_rate} and~\ref{tab:hyperparameters_100d_sigma}.

For the smoothed gradient computation we use $1000$ Monte Carlo samples in MC-GSmoothGD and $5$ Gauss--Hermite quadrature points in DGS across all experiments.
The value of the momentum in NAG is set to $\beta = 0.5$ to offer a reasonable compromise since there seems to be no conventional value and the practical recommendations typically range between $0.1$ and $0.9$.
All unspecified hyperparameters are set to their default values selected by the algorithm's authors.
Note that BFGS and CG use adaptive learning rates, thus these algorithms are not listed in Table~\ref{tab:hyperparameters_100d_learning_rate}.

\begin{figure}[t]
    \centering
    \begin{subfigure}{\linewidth}
        \includegraphics[width=.32\linewidth]{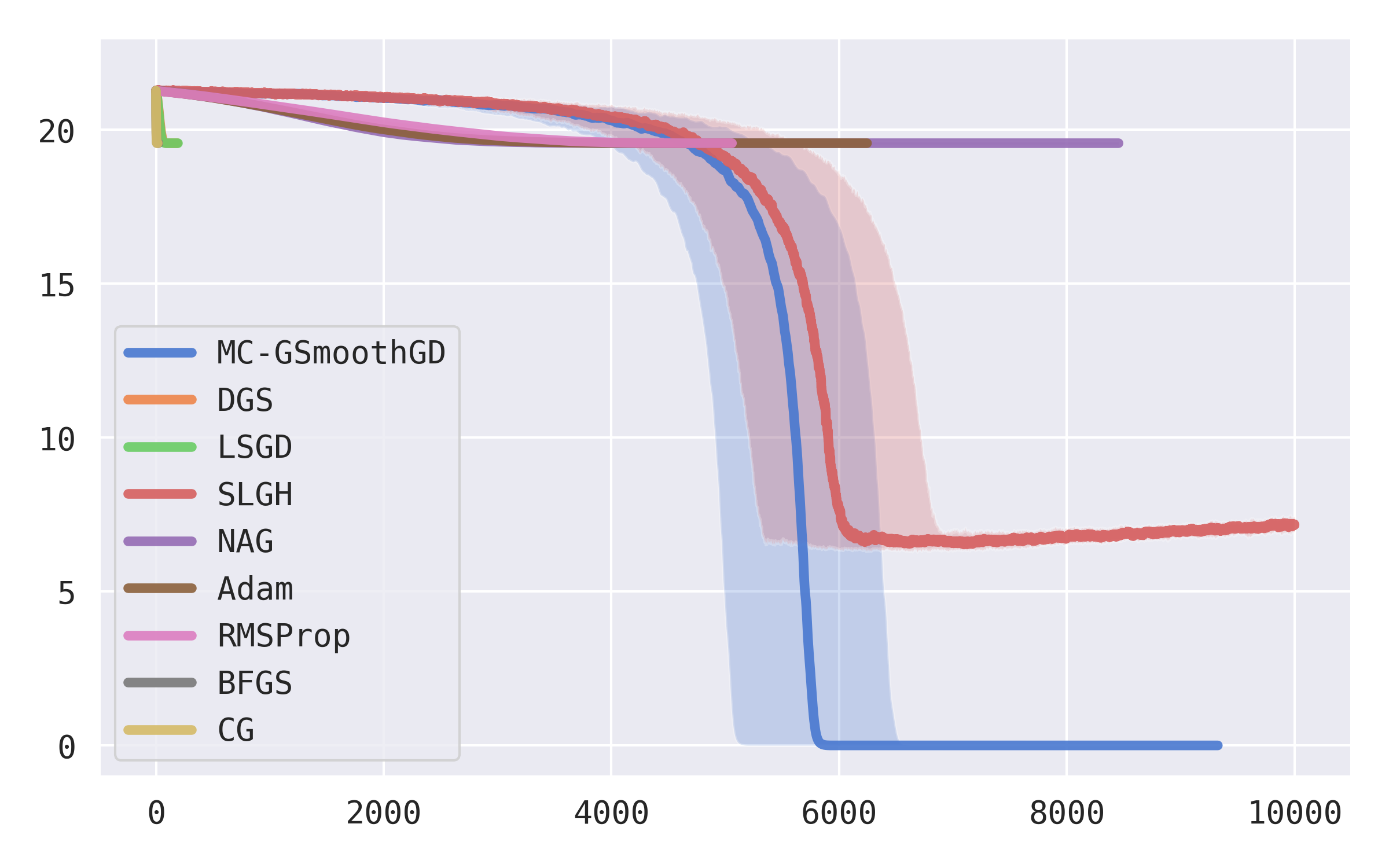}
        \includegraphics[width=.32\linewidth]{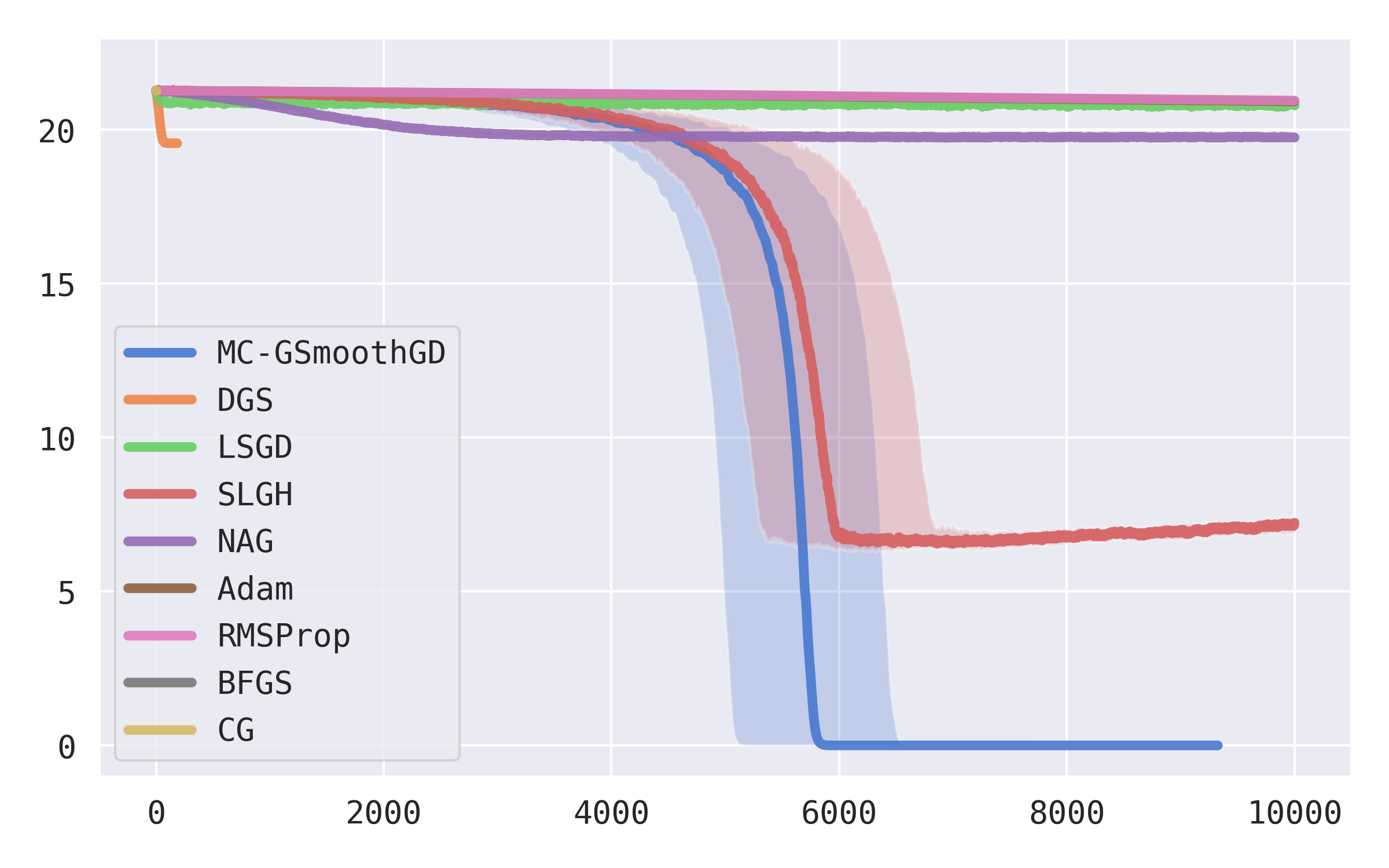}
        \includegraphics[width=.32\linewidth]{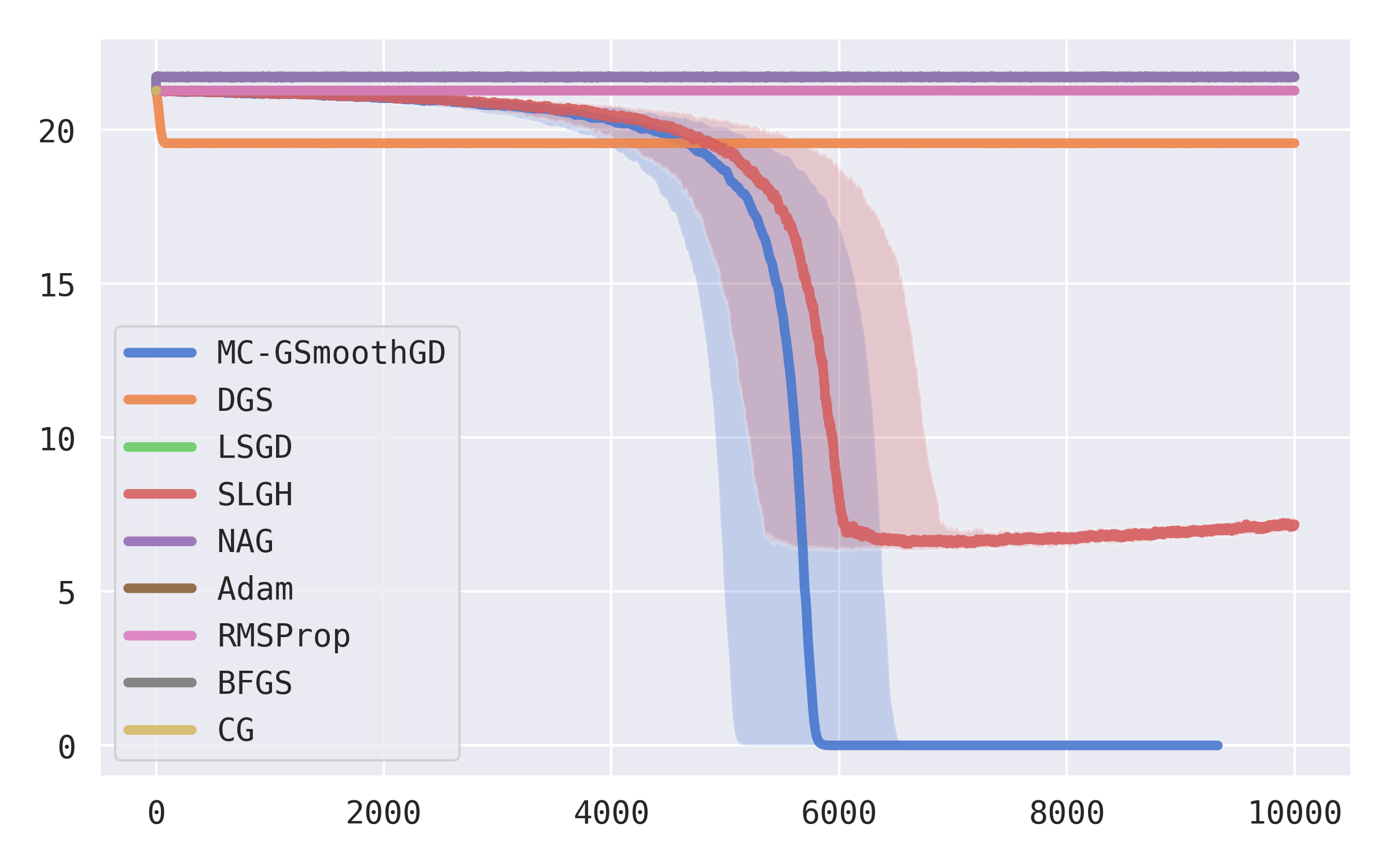}
        \caption{Ackley function}
    \end{subfigure}
    \\
    \begin{subfigure}{\linewidth}
        \includegraphics[width=.32\linewidth]{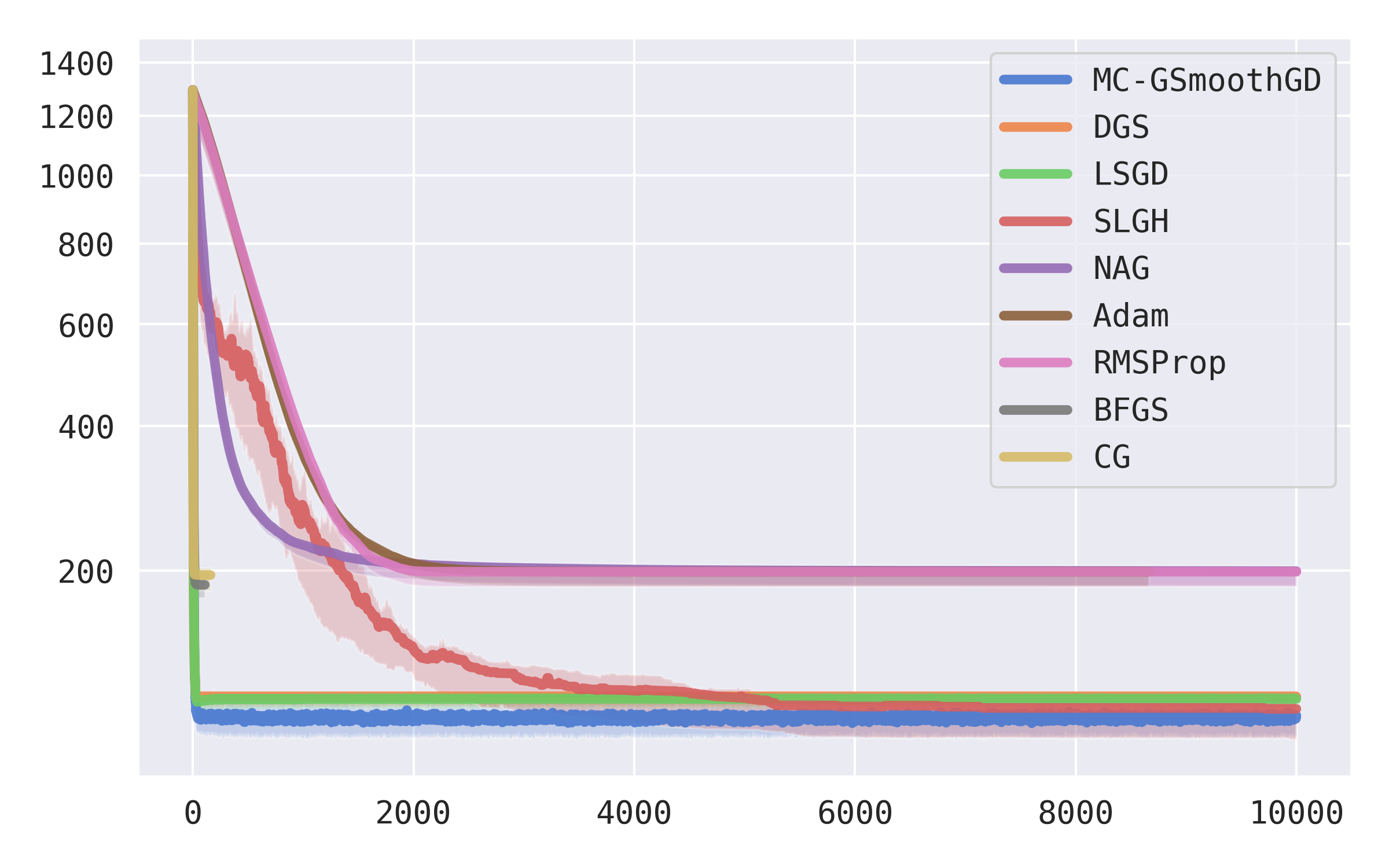}
        \includegraphics[width=.32\linewidth]{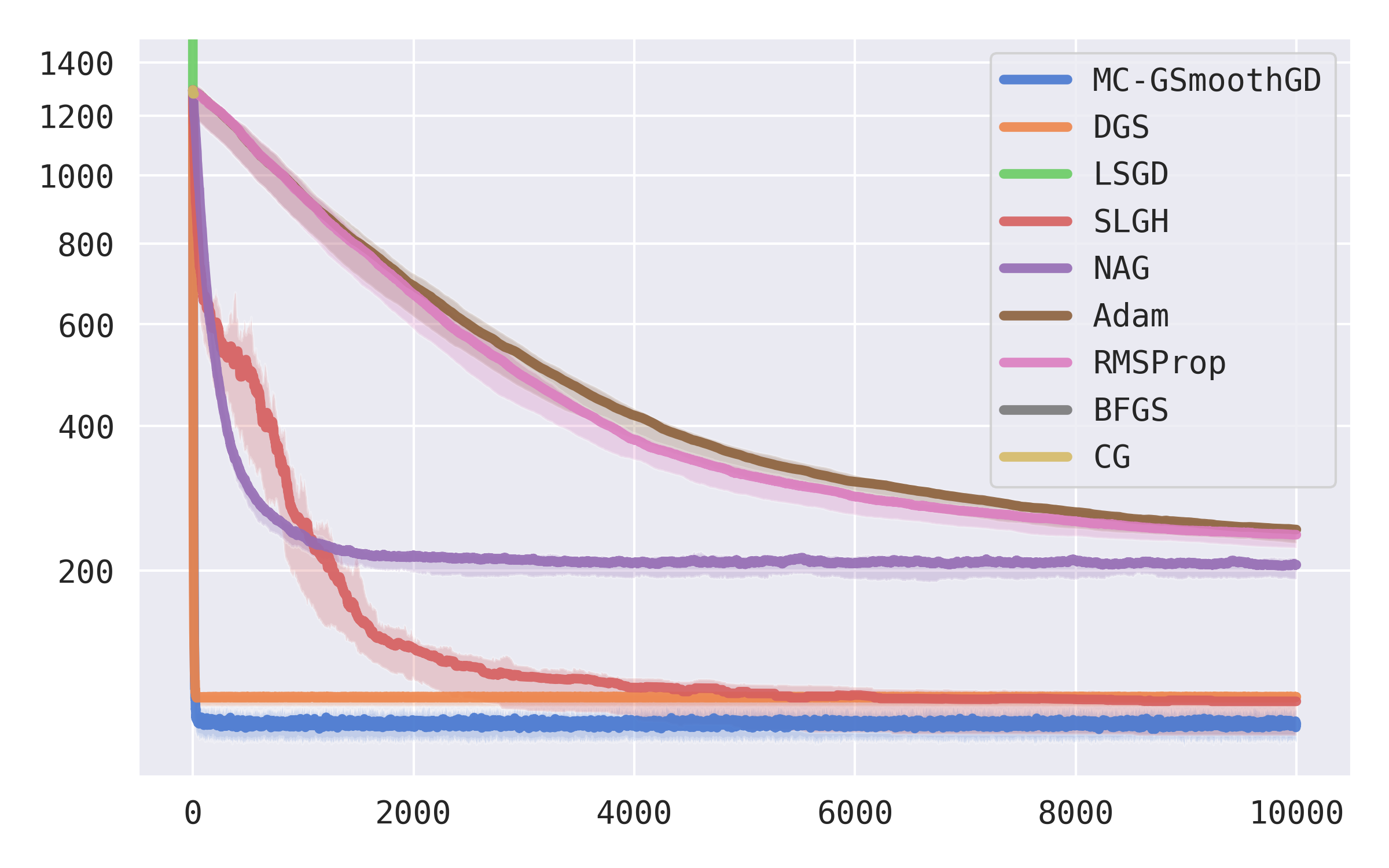}
        \includegraphics[width=.32\linewidth]{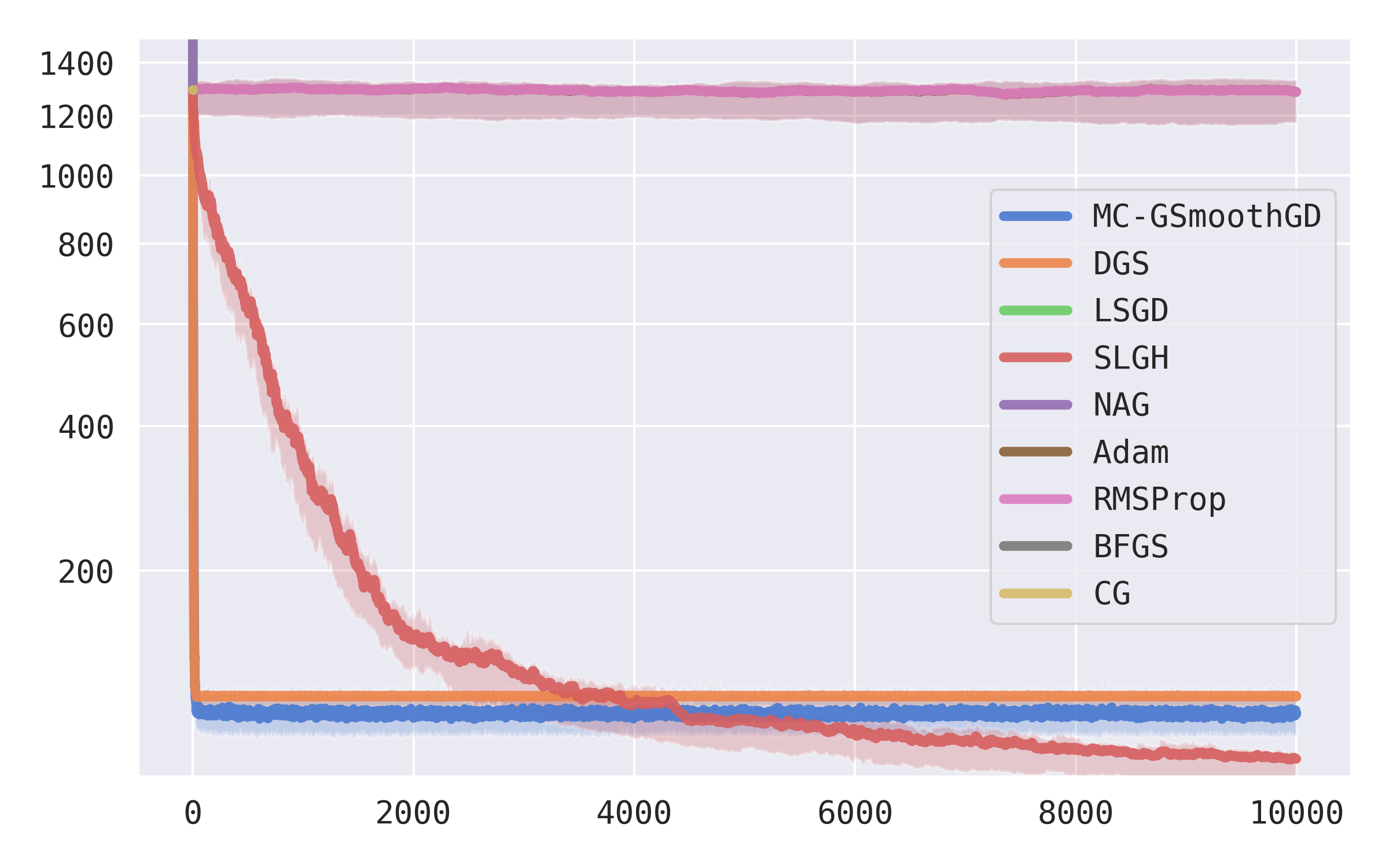}
        \caption{Levy function}
    \end{subfigure}
    \\
    \begin{subfigure}{\linewidth}
        \includegraphics[width=.32\linewidth]{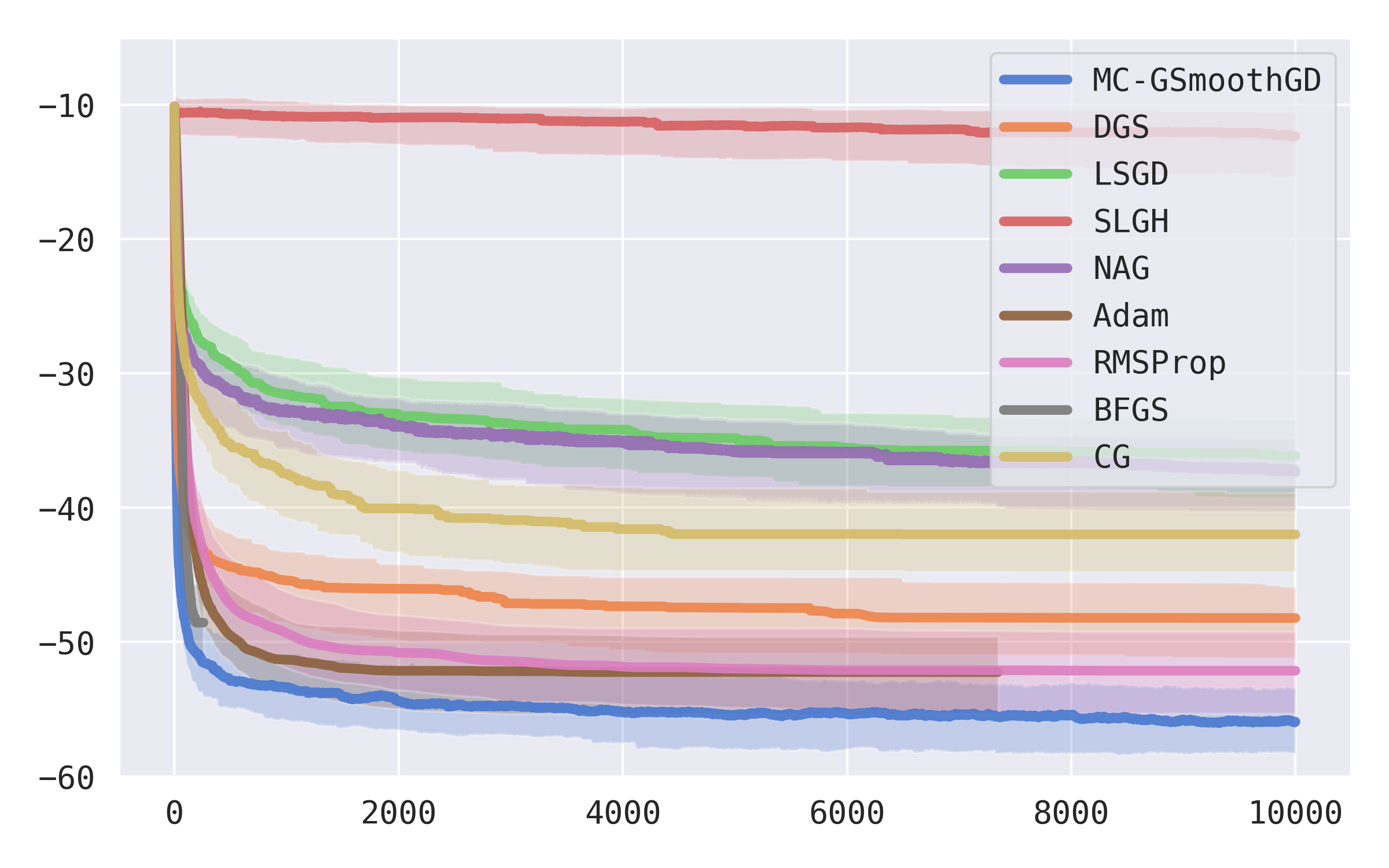}
        \includegraphics[width=.32\linewidth]{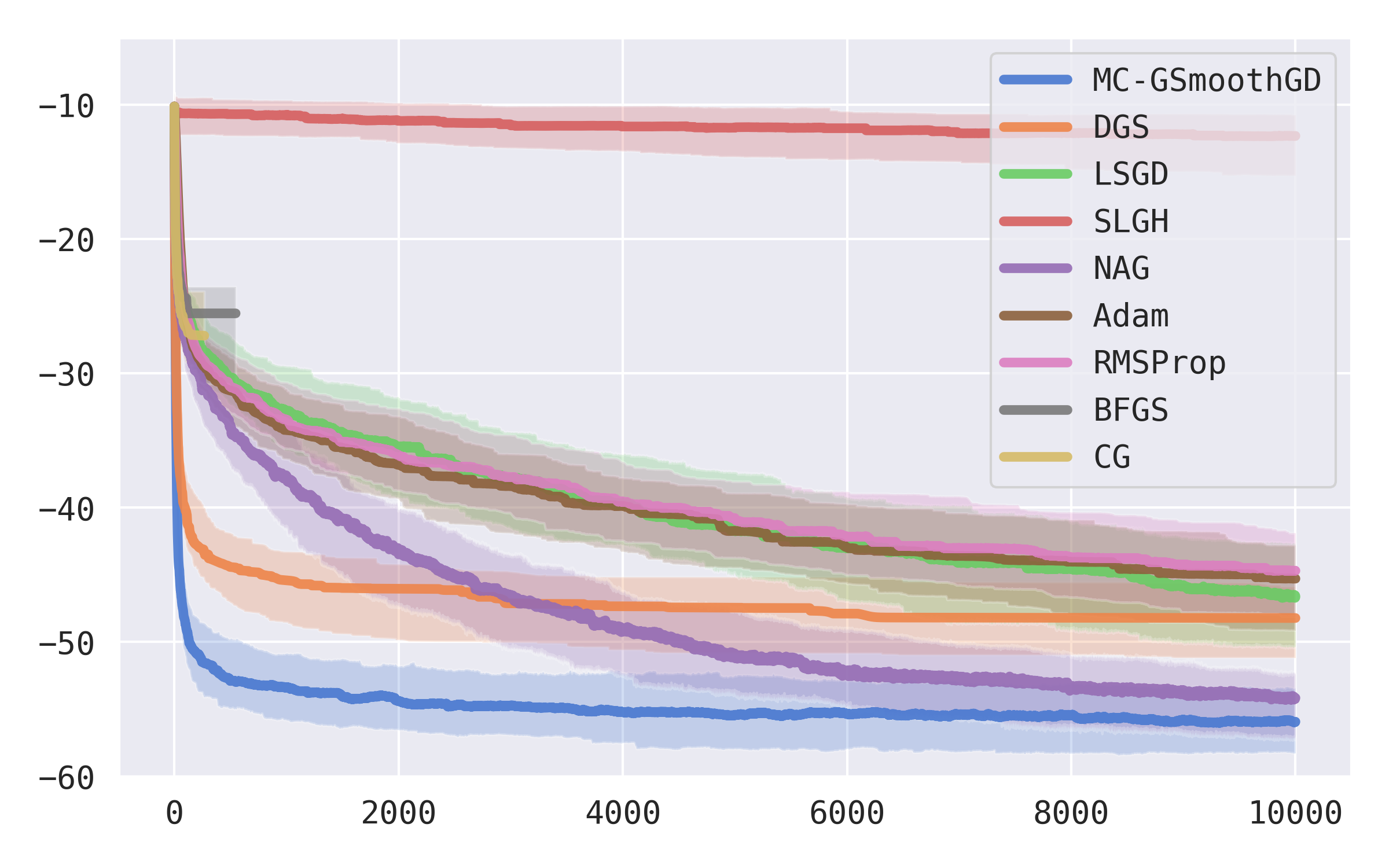}
        \includegraphics[width=.32\linewidth]{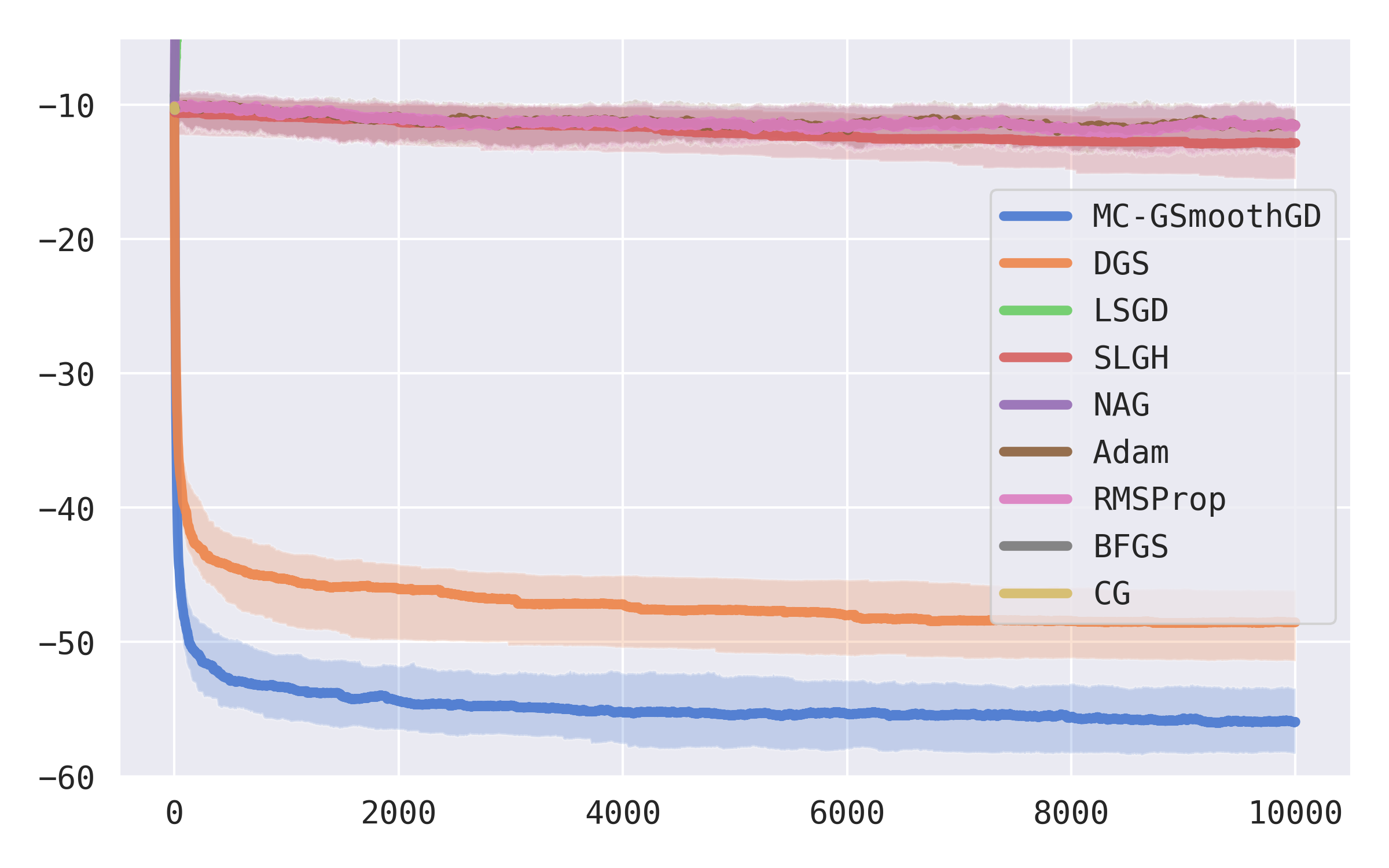}
        \caption{Michalewicz function}
    \end{subfigure}
    \\
    \begin{subfigure}{\linewidth}
        \includegraphics[width=.32\linewidth]{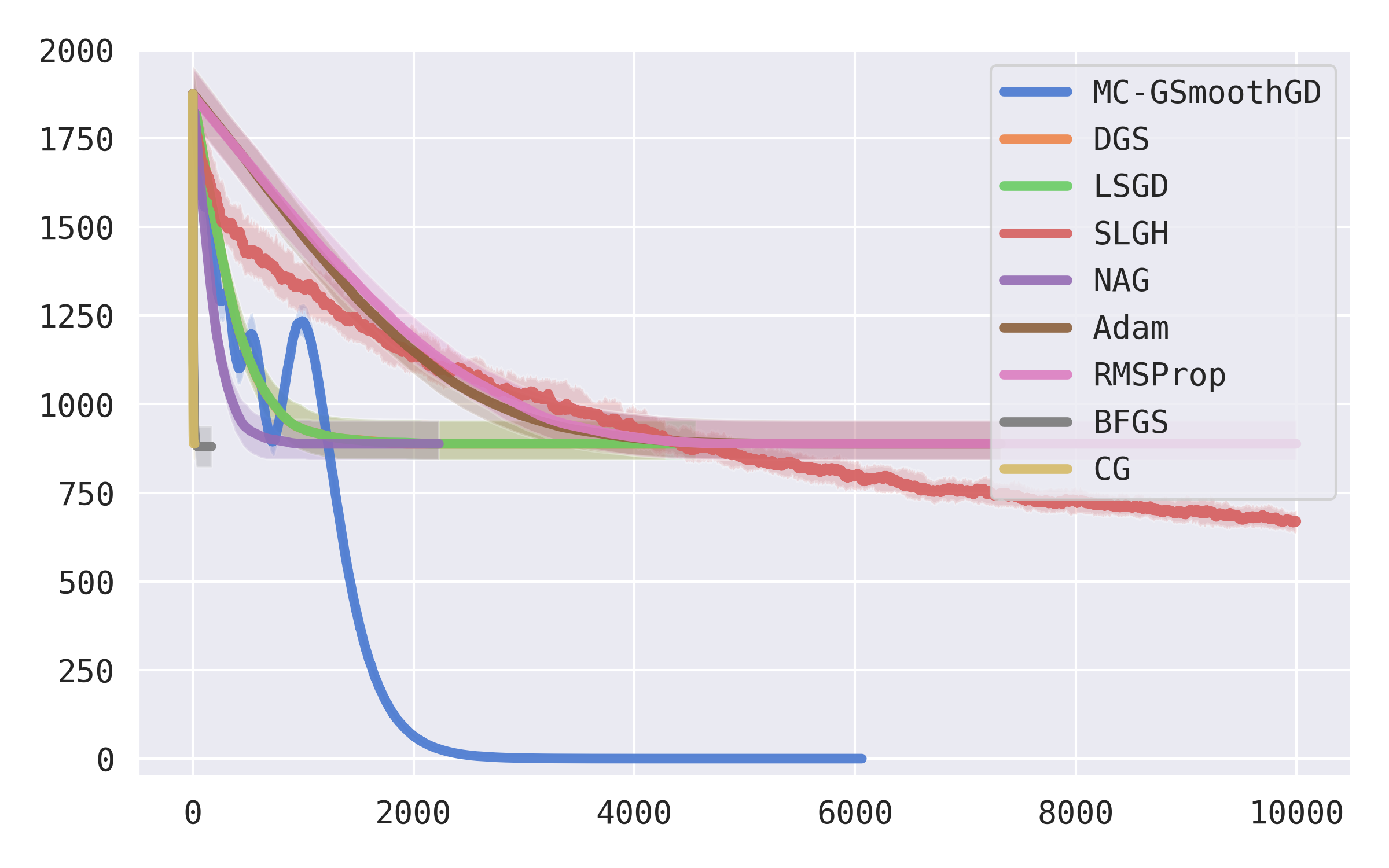}
        \includegraphics[width=.32\linewidth]{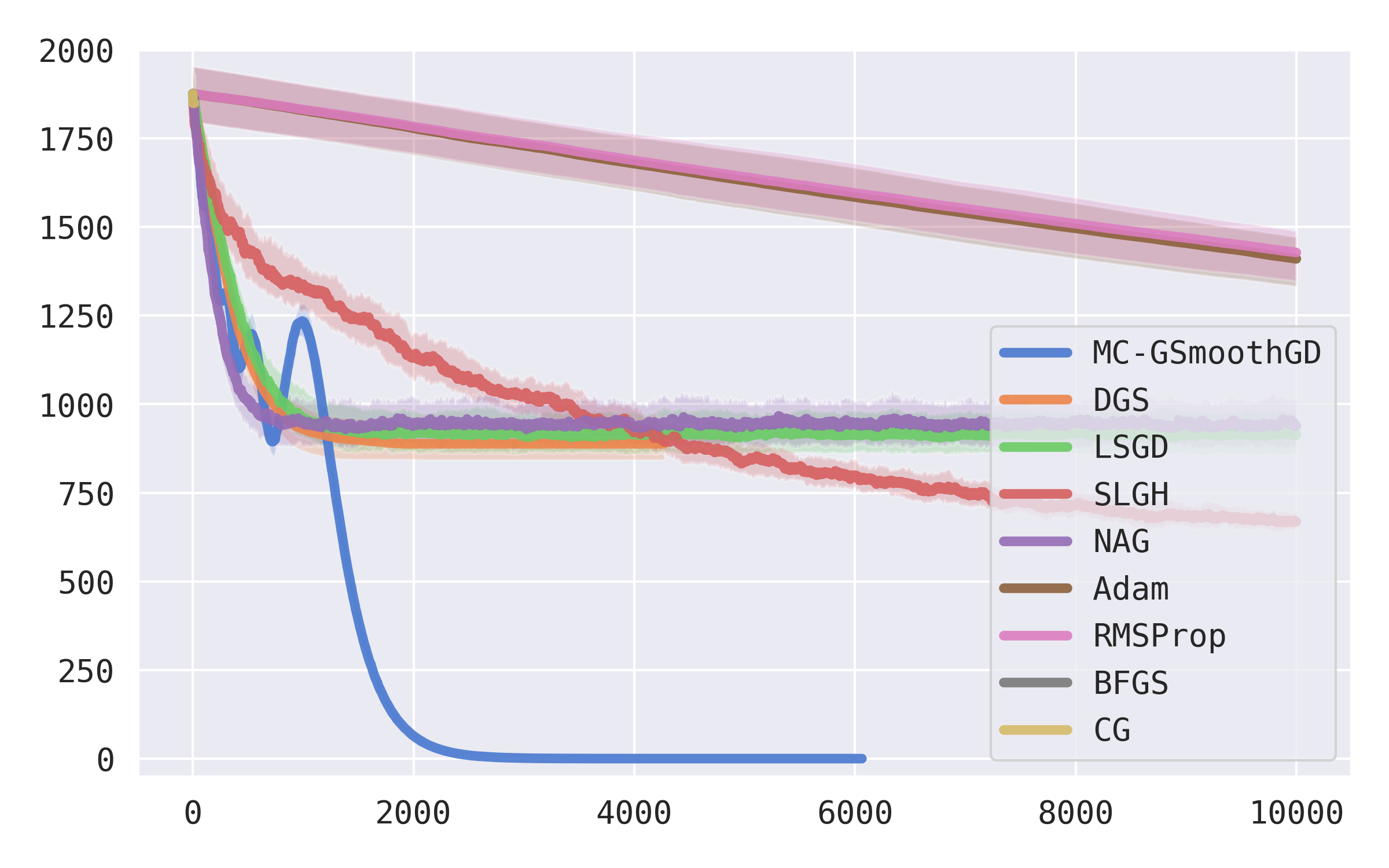}
        \includegraphics[width=.32\linewidth]{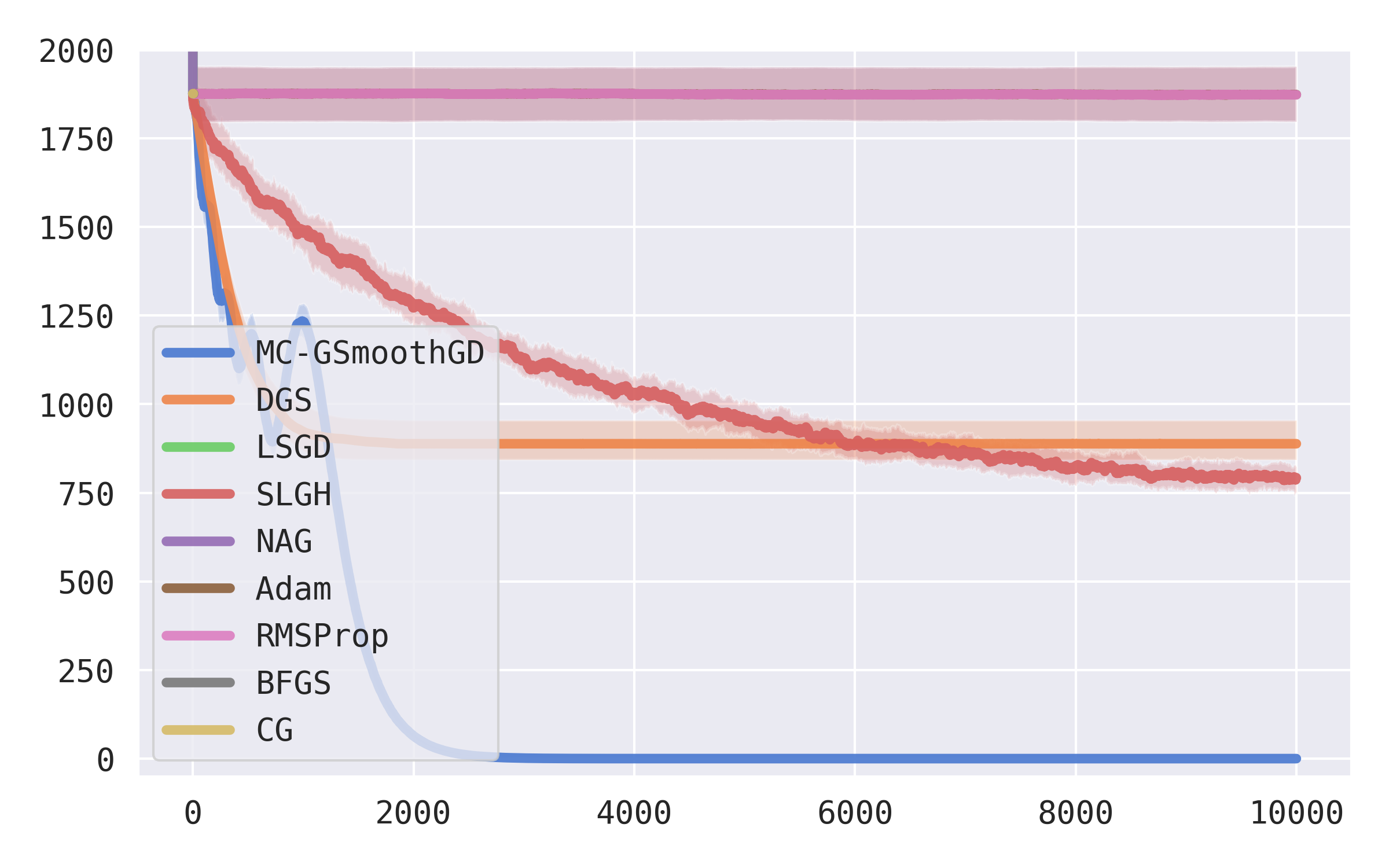}
        \caption{Rastrigin function}
    \end{subfigure}
    \\
    \begin{subfigure}{\linewidth}
        \includegraphics[width=.32\linewidth]{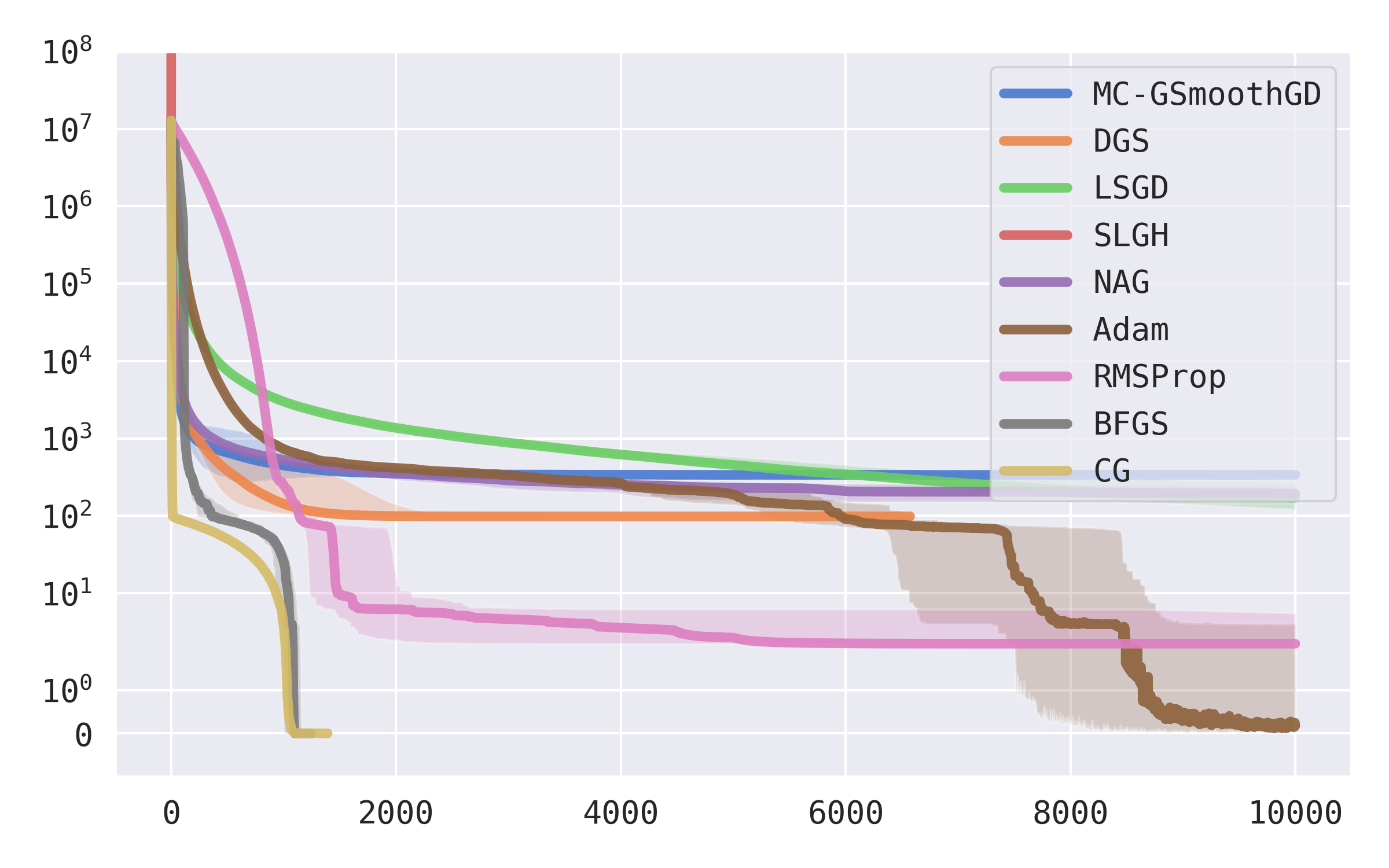}
        \includegraphics[width=.32\linewidth]{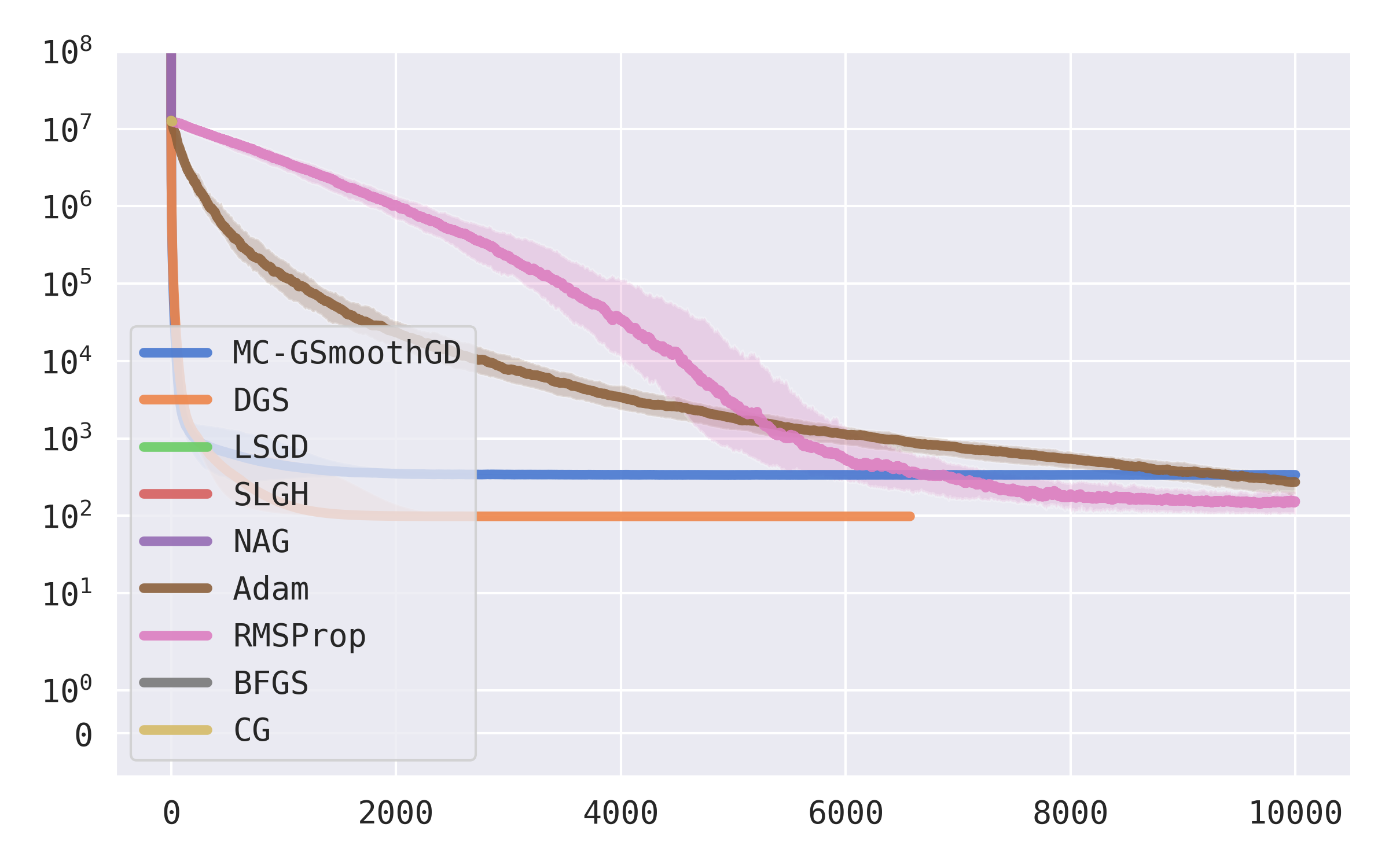}
        \includegraphics[width=.32\linewidth]{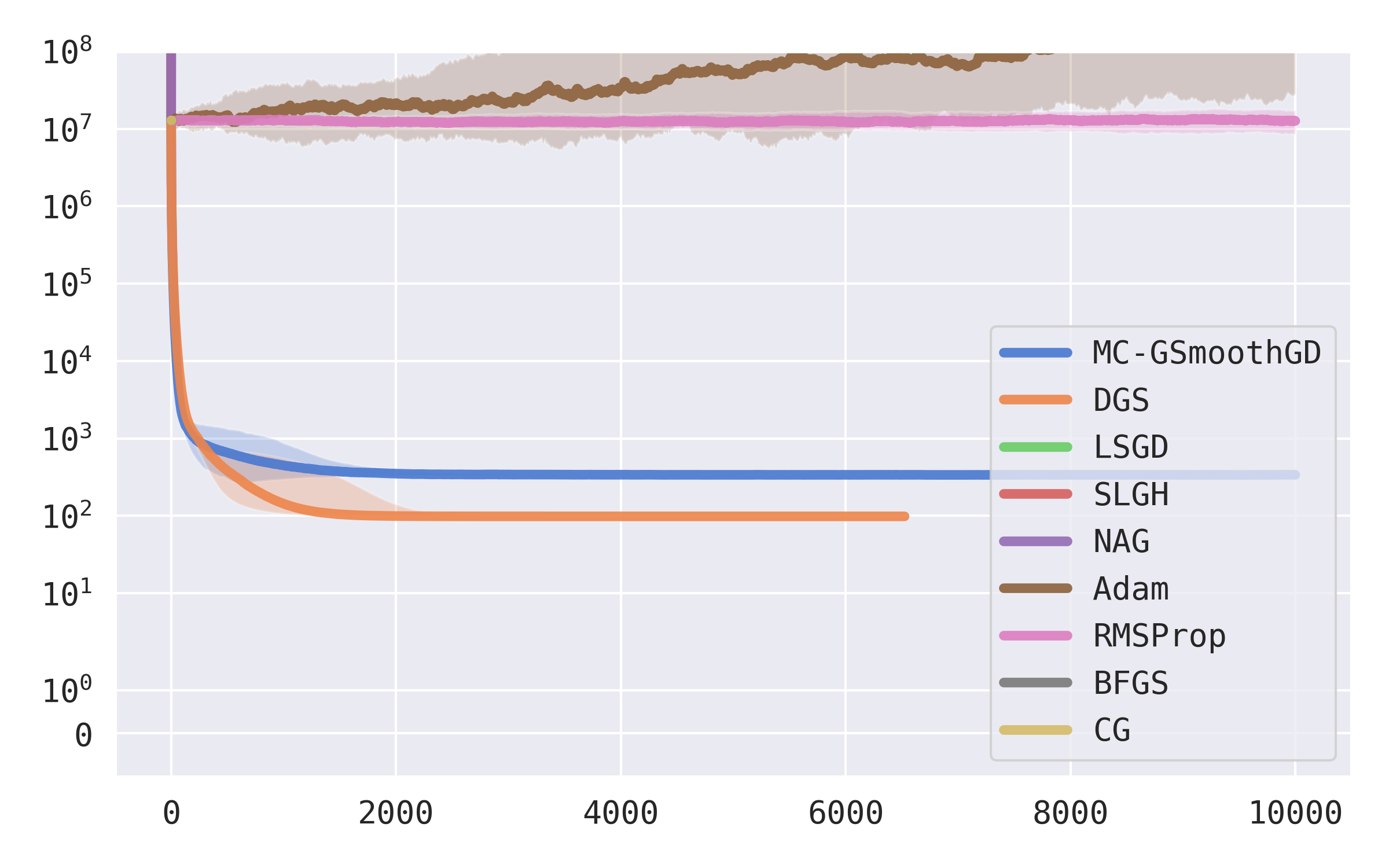}
        \caption{Rosenbrock function}
    \end{subfigure}
    \\
    \begin{subfigure}{\linewidth}
        \includegraphics[width=.32\linewidth]{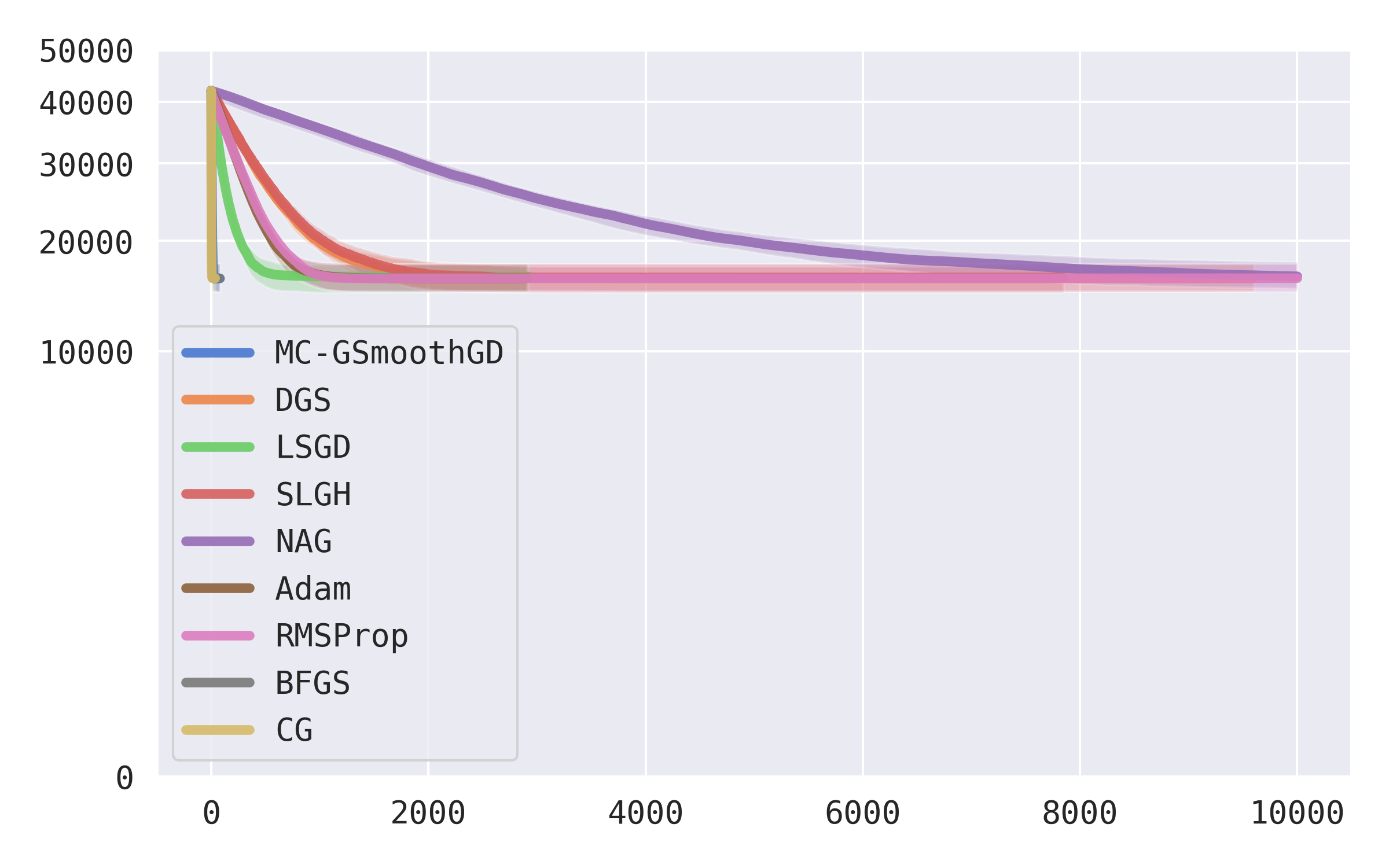}
        \includegraphics[width=.32\linewidth]{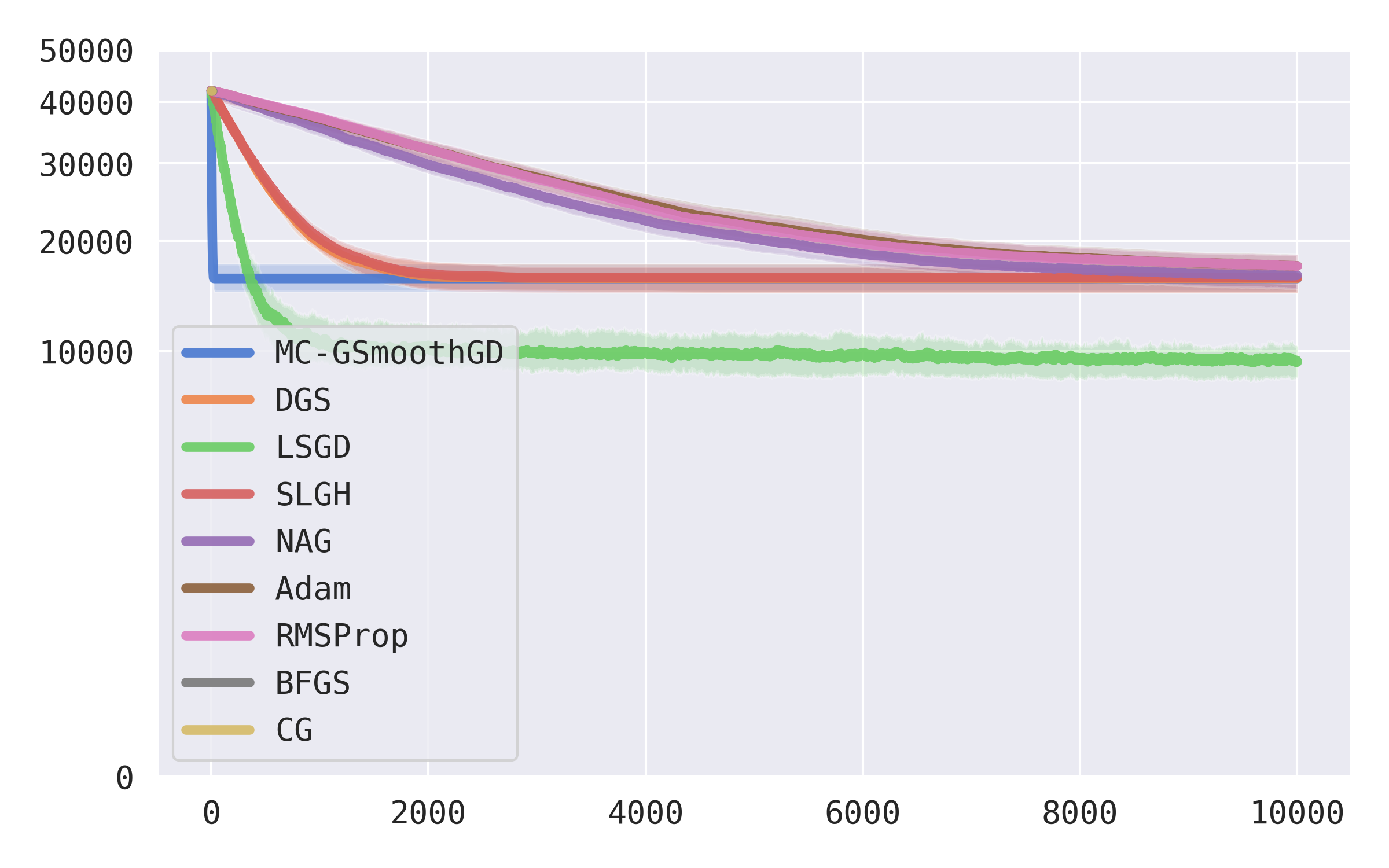}
        \includegraphics[width=.32\linewidth]{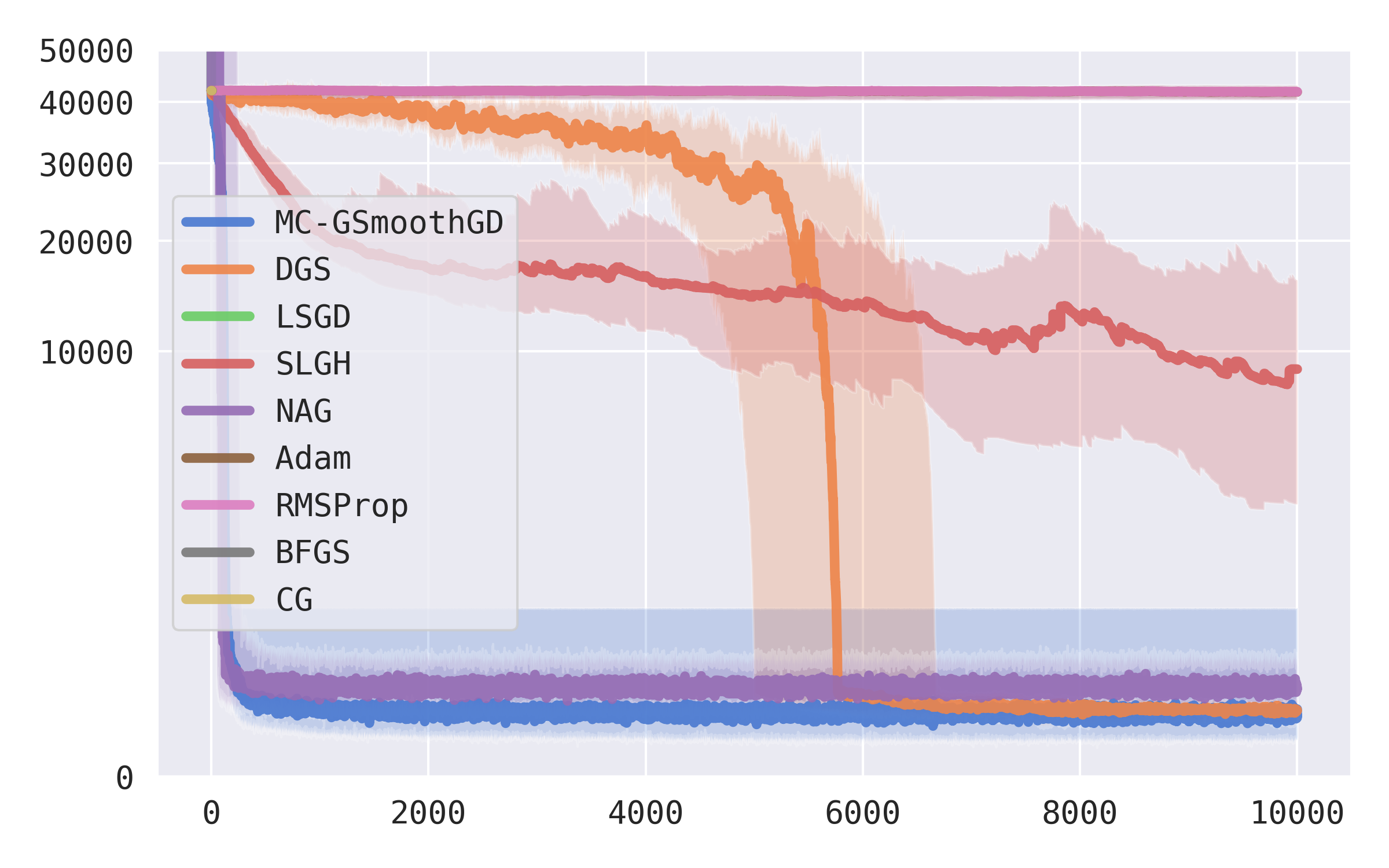}
        \caption{Schwefel function}
    \end{subfigure}
    \caption{Numerical experiments on $100$-dimensional test functions with various relative noise levels:
             no noise (left), $0.000001$\% (middle), and $0.01$\% (right).
             }
    \label{fig:numerics_100d_noise}
\end{figure}

\begin{table}
    \centering\small
    \begin{tabular}{lcccccc}
        \toprule
        & \multicolumn{6}{c}{Target function}
        \\\cmidrule(lr){2-7}
        Algorithm & Ackley & Levy & Michalewicz & Rastrigin & Rosenbrock & Schwefel
        \\\midrule
        MC-GSmoothGD & \texttt{1e-0} & \texttt{1e-1} & \texttt{1e-4} & \texttt{1e-3} & \texttt{1e-5} & \texttt{1e-0}
        \\
        DGS & \texttt{1e-1} & \texttt{1e-1} & \texttt{1e-4} & \texttt{1e-5} & \texttt{1e-5} & \texttt{1e-2}
        \\
        LSGD & \texttt{1e-1} & \texttt{1e-1} & \texttt{1e-5} & \texttt{1e-5} & \texttt{1e-5} & \texttt{1e-1}
        \\
        SLGH & \texttt{1e-0} & \texttt{1e-3} & \texttt{1e-6} & \texttt{1e-4} & \texttt{1e-5} & \texttt{1e-2}
        \\
        NAG & \texttt{1e-3} & \texttt{1e-4} & \texttt{1e-5} & \texttt{1e-5} & \texttt{1e-5} & \texttt{1e-3}
        \\
        Adam & \texttt{1e-4} & \texttt{1e-3} & \texttt{1e-4} & \texttt{1e-4} & \texttt{1e-1} & \texttt{1e-1}
        \\
        RMSProp & \texttt{1e-4} & \texttt{1e-3} & \texttt{1e-4} & \texttt{1e-4} & \texttt{1e-2} & \texttt{1e-1}
        \\\bottomrule
    \end{tabular}
    \caption{Values of the learning rate $\lambda$ for optimization algorithms on $100$-dimensional test functions.
             See Figure~\ref{fig:numerics_100d_noise} for the performance of the corresponding algorithms.}
    \label{tab:hyperparameters_100d_learning_rate}
\end{table}

\begin{table}
    \centering\small
    \begin{tabular}{lcccccc}
        \toprule
        & \multicolumn{6}{c}{Target function}
        \\\cmidrule(lr){2-7}
        Algorithm & Ackley & Levy & Michalewicz & Rastrigin & Rosenbrock & Schwefel
        \\\midrule
        MC-GSmoothGD & \texttt{1e-0} & \texttt{1e-2} & \texttt{1e-2} & \texttt{1e-0} & \texttt{1e-0} & \texttt{1e-3}
        \\
        DGS & \texttt{1e-1} & \texttt{1e-1} & \texttt{1e-2} & \texttt{1e-3} & \texttt{1e-0} & \texttt{1e-3}
        \\
        LSGD & \texttt{1e-3} & \texttt{1e-2} & \texttt{1e-1} & \texttt{1e-3} & \texttt{1e-0} & \texttt{1e-0}
        \\
        SLGH & \texttt{1e-0} & \texttt{1e-1} & \texttt{1e-3} & \texttt{1e-2} & \texttt{1e-0} & \texttt{1e-3}
        \\\bottomrule
    \end{tabular}
    \caption{Values of the smoothing parameter $\sigma$ for the smoothing-based optimization algorithms on $100$-dimensional test functions.}
    \label{tab:hyperparameters_100d_sigma}
\end{table}


\section{Conclusions}
\label{sec:conclusions}

This effort delves into the convergence analysis of a category of smoothing-based gradient descent methods when applied to high-dimensional non-convex optimization challenges. Gaussian smoothing is harnessed to define the nonlocal gradient, which significantly aids gradient descent in navigating away from local minima and enhances overall performance in non-convex optimization tasks. Furthermore, this work provides rigorous theoretical error estimates on the convergence rate of GSmoothGD iterates, taking into account the influence of function convexity, smoothness, input dimension, and the Gaussian smoothing radius.  To address curse of dimensionality, we numerically approximate the $d$-dimensional GSmoothGD nonlocal gradient using MC sampling, backed by a theory that demonstrates convergence regardless of function smoothness and dimensionality. To combat local minima's influence, the paper discusses strategies for updating the smoothing parameter, making global minima more attainable.
These methods, in their bid to alleviate high-frequency noise, minor fluctuations, and rapid variations in the computation of descent directions, replace conventional (local) gradients with (nonlocal) approximations that aim to preserve the overall large-scale structure or features of the loss landscape.
In our investigation of optimization techniques, we present empirical evidence that highlights the remarkable stability of smoothing-based optimization when subjected to external perturbations.
Our findings reveal that algorithms employing Gaussian smoothing mechanisms exhibit an extraordinary resilience to the presence of noise, a characteristic that holds significant practical utility in real-world applications.
Finally, future endeavors will focus on extending the theory and numerical approaches for Gaussian smoothing to improve other popular approaches for high-dimensional non-convex optimization, including stochastic gradient descent (GSmoothSGD), as well as stochastic variance reduced gradient (GSmoothSVRG).  These approaches aim at reducing noise and uncertainty in the gradient approximation and, thereby, enhancing the effectiveness of SGD as well as SVRG.

\clearpage
\appendix
\section{Proofs of background results}
\label{app:notation}
\subsection*{Proof of Lemma~\ref{lem:fsigmalsmoothandconvex}}
\label{app:fsigmalsmoothandconvex}
%
For the proofs of Lemma \ref{lem:fsigmalsmoothandconvex}(a) and (b) please see~\cite[p. 533]{nesterov2017random}.

For the first part of (c), the fact that $f_{\sigma}$ is $M$-Lipschitz is stated as the 2nd bullet on pg. 533 of \cite{nesterov2017random}.
For the second half,
\begin{align}
\begin{split}
    \|\nabla f_{\sigma}(\bmx)-\nabla f_{\sigma}(\bmy)\|
    &\leq\frac{2}{\sigma\pi^{\frac{d}{2}}}\int_{\mathbb{R}^d}|f(\bmx+\sigma \bmu)-f(\bmy+\sigma \bmu)|\|\bmu\|e^{-\|\bmu\|^2}\;d\bmu\\
    &\leq\frac{2M\|\bmx-\bmy\|}{\sigma\pi^{\frac{d}{2}}}\int_{\mathbb{R}^d}\|\bmu\|e^{-\|\bmu\|^2}\;d\bmu\\
    &=\frac{M\sqrt{2d}}{\sigma}\|\bmx-\bmy\|,
\end{split}
\end{align}
using Lemma~\ref{lem:gradient_goes_into_convolution} to pass the gradient into the integral.


\subsection*{Proof of Lemma~\ref{lem:fsigmagreaterthanf}}
\label{app:fsigmagreaterthanf}
%
For the proof of (a), 
let $\sigma>0$ and suppose $m=\inf_{\bmx\in\mathbb{R}^d}f(\bmx)>-\infty$.
Assume that there is some $\bmx_0\in\mathbb{R}^d$ so that $f_{\sigma}(\bmx_0)=m$. Then
\begin{equation}
    0
    = f_{\sigma}(\bmx_0)-m
    =\frac{1}{\pi^{\frac{d}{2}}}\int_{\mathbb{R}^d}\Big(f(\bmx_0+\sigma \bmu)-m\Big)e^{-\|\bmu\|^2}\;d\bmu.
\end{equation}
Since $(f(\bmx_0+\sigma \bmu)-m)e^{-\|\bmu\|^2}\geq 0$ and $e^{-\|\bmu\|^2}\not=0$, we have $f(\bmx_0+\sigma \bmu)-m=0$ for all $\bmu\in\mathbb{R}^d$.
Further for any $\bmx\in\mathbb{R}^d$ we can find a $\bmu\in\mathbb{R}^d$ so that $\bmx=\bmx_0+\sigma \bmu$, which shows $f(\bmx)=m$ for all $\bmx\in\mathbb{R}^d$.
Thus $f$ is constant. This same approach obviously holds for a maximizer of $f$ as well.

For the proof of (b),
since $f(\bmx+\sigma \bmu)-m\geq 0$, we have
\begin{equation}
    f_{\sigma}(\bmx)-m
    =\frac{1}{\pi^{\frac{d}{2}}}\int_{\mathbb{R}^d}\Big(f(\bmx+\sigma \bmu)-m\Big)e^{-\|\bmu\|^2}\;d\bmu
    \geq 0.
\end{equation}
This means $f_{\sigma}(\bmx)\geq m$. Repeating this for $M-f(\bmx+\sigma \bmu)\geq 0$, gives the other result.


\subsection*{Proof of Lemma~\ref{lem:chainofsmoothing}}
\label{app:chainofsmoothing}
We have that
\begin{align}
\begin{split}
    (f_{\sigma})_{\tau}(\bmx)
    &=\frac{1}{\pi^{\frac{d}{2}}}\int_{\mathbb{R}^d}\left(\frac{1}{\pi^{\frac{d}{2}}}\int_{\mathbb{R}^d}f((\bmx+\tau \bmv)+\sigma \bmu)e^{-\|\bmu\|^2}\;d\bmu\right)e^{-\|\bmv\|^2}\;d\bmv\\
    &=\frac{1}{\pi^d}\int_{\mathbb{R}^d}\int_{\mathbb{R}^d}f(\bmx+\tau \bmv+\sigma \bmu)e^{-(\|\bmu\|^2+\|\bmv\|^2)}\;d\bmu d\bmv.
\end{split}
\end{align}

Using the change of variables\footnote{We are free to use a change of variables since $f$ is bounded below.}  $\bms=\sigma \bmu + \tau \bmv$ and $ \bmt=\tau \bmu - \sigma \bmv$
we end up with
\begin{align}
\begin{split}
    (f_{\sigma})_{\tau}(\bmx)
    &=\frac{1}{\pi^d\eta^{2d}}\int_{\mathbb{R}^d}\int_{\mathbb{R}^d}f(\bmx+\bms)e^{-\frac{1}{\eta^2}(\|\bms\|^2+\|t\|^2)}\;d\bmt d\bms\\
    &=\frac{1}{\pi^{\frac{d}{2}}}\int_{\mathbb{R}^d}f(\bmx+\eta \bmw)e^{-\|\bmw\|^2}\;d\bmw\\
    &=f_{\eta}(\bmx),
\end{split}
\end{align}
which completes the proof.

\subsection*{Proof of Lemma~\ref{lem:differentsmoothingvalues}}
\label{app:differentsmoothingvalues}
First, let $f:\mathbb{R}^d\to\mathbb{R}$ be $L$-smooth.
From Lemma~\ref{lem:chainofsmoothing}, if $\eta = \sqrt{\tau^2-\sigma^2}$, then
$
    f_{\tau}(\bmx)=(f_{\sigma})_{\eta}(\bmx).
$
So,
\begin{align}
\begin{split}
    |f_{\tau}(\bmx)-f_{\sigma}(\bmx)|
    =|(f_{\sigma})_{\eta}(\bmx)-f_{\sigma}(\bmx)|.
\end{split}
\end{align}
As such, it suffices to show
\begin{equation}
    |f_{\sigma}(\bmx)-f(\bmx)|\leq\frac{\sigma^2Ld}{4}.
\end{equation}

To this end, note that
\begin{align}
\begin{split}
    f_{\sigma}(\bmx)-f(\bmx)
    &=
    \frac{1}{\pi^{\frac{d}{2}}}\int_{\mathbb{R}^d}\Big(f(\bmx+\sigma \bmu)-f(\bmx)-\langle\nabla f(\bmx),\sigma \bmu\rangle\Big)e^{-\|\bmu\|^2}d\bmu.
\end{split}
\end{align}
Therefore, since $f$ is $L$-smooth,
\begin{align}
\begin{split}
    |f_{\sigma}(\bmx)-f(\bmx)|
    &\leq \frac{1}{\pi^{\frac{d}{2}}}\int_{\mathbb{R}^d}\Big|f(\bmx+\sigma \bmu)-f(\bmx)-\langle\nabla f(\bmx),\sigma \bmu\rangle\Big|e^{-\|\bmu\|^2}d\bmu\\
    &\leq\frac{1}{\pi^{\frac{d}{2}}}\int_{\mathbb{R}^d}\frac{L}{2}\|\sigma \bmu\|^2e^{-\|\bmu\|^2}d\bmu\\
    &=\frac{\sigma^2L}{2\pi^{\frac{d}{2}}}\int_{\mathbb{R}^d}\|\bmu\|^2e^{-\|\bmu\|^2}d\bmu
    =\frac{\sigma^2L}{2\pi^{\frac{d}{2}}}\left(\frac{d\pi^{\frac{d}{2}}}{2}\right)
    =\frac{\sigma^2Ld}{4}.
\end{split}
\end{align}

To prove part (b), let $f$ be $M$-Lipschitz.
Then
\begin{equation}
    |f_{\tau}(\bmx)-f_{\sigma}(\bmx)|
    \leq\frac{1}{\pi^{\frac{d}{2}}}\int_{\mathbb{R}^d}|f(\bmx+\tau \bmu)-f(\bmx+\sigma \bmu)|e^{-\|\bmu\|^2}d\bmu
    \leq M|\tau-\sigma|\sqrt{\frac{d}{2}}.
\end{equation}

\subsection*{Proof of Corollary~\ref{cor:fminusfsigma}}
\label{app:lem:fminusfsigma}
Let $f:\mathbb{R}^d\to\mathbb{R}$.
By Lemma~\ref{lem:fsigmagreaterthanf}, we have that
$
    f_{\sigma}(\sstar)\geq f(\xstar).
$
%
Suppose $f$ is $L$-smooth and $\sigma\geq 0$. Since $\sstar$ minimizes $f_{\sigma}$ then using Lemma~\ref{lem:differentsmoothingvalues}, we have
\begin{align}
    f_{\sigma}(\sstar)-f(\xstar)\leq f_{\sigma}(\xstar)-f(\xstar)\leq\frac{\sigma^2Ld}{4}.
\end{align}
If $f$ is $M$-Lipschitz instead, Lemma~\ref{lem:differentsmoothingvalues} again gives
\begin{align}
    f_{\sigma}(\sstar)-f(\xstar)\leq f_{\sigma}(\xstar)-f(\xstar)\leq M\sigma\sqrt{\frac{d}{2}}.
\end{align}

\subsection*{Proof of Lemma~\ref{lem:finl1}}
\label{app:lem:finl1}
\begin{proof}
Suppose $f$ is bounded below by $m$.
Then $0\leq f(\bmx)-m$ for all $\bmx\in\mathbb{R}$.
This means
\begin{align}
\begin{split}
    \frac{1}{\pi^{\nicefrac{d}{2}}}\int_{\mathbb{R}^d}|f(\bmx+\sigma\bmu)|e^{-\|u\|^2}\;d\bmu
    &=\frac{1}{\pi^{\nicefrac{d}{2}}}\int_{\mathbb{R}^d}|f(\bmx+\sigma\bmu)-m+m|e^{-\|u\|^2}\;d\bmu\\
    &\leq\frac{1}{\pi^{\nicefrac{d}{2}}}\int_{\mathbb{R}^d}\big(|f(\bmx+\sigma\bmu)-m|+|m|\big)e^{-\|u\|^2}\;d\bmu\\
    &=\frac{1}{\pi^{\nicefrac{d}{2}}}\int_{\mathbb{R}^d}\big(f(\bmx+\sigma\bmu)-m\big)e^{-\|u\|^2}\;d\bmu + |m|\\
    &=f_{\sigma}(\bmx)-m+|m|,
\end{split}
\end{align}
which is a real number for any $\bmx\in\mathbb{R}^d$.
Thus, $f\in L^1(\mathbb{R}^d,k_{\sigma})$.
\end{proof}

\subsection*{Proof of Lemma~\ref{lem:gradient_goes_into_convolution}}
\label{app:lem:gradient_goes_into_convolution}
\begin{proof}
We have that
\begin{align}
\begin{split}
    \frac{f_{\sigma}(\bmx+he_i)-f_{\sigma}(\bmx)}{h}
    &=\int_{\mathbb{R}^d}f(\bmx+\sigma\bmu)\left(\frac{e^{-\|\bmu-\frac{h}{\sigma}e_i\|^2}-e^{-\|\bmu\|^2}}{h}\right)\;d\bmu\\
    &=\frac{-2}{\sigma^2}\int_{\mathbb{R}^d}f(\bmx+\sigma\bmu)(x_i+ch)e^{-\left(x_i+\frac{ch}{\sigma}\right)^2}e^{-\sum_{j\not= i}x_j^2}\;d\bmu
\end{split}
\end{align}
for some $c\in(0,1)$ by the MVT.
Now we will show that we can switch the limit and the integral.
Note that
\begin{align}
    (x_i+ch)e^{-\nicefrac{(x_i+ch)^2}{\sigma^2}}e^{-\nicefrac{x_i^2}{\tau^2}}
    &=(x_i+ch)e^{\left(\frac{\sigma^2-\tau^2}{\sigma^2\tau^2}\right)\left(x-\frac{h\sigma^2\tau^2}{\sigma^2-\tau^2}\right)^2}e^{\frac{h}{\sigma^2}\left(1+\frac{\tau^2}{\tau^2-\sigma^2}\right)}.
\end{align}
So, for $\tau > \sigma$, since
\begin{align}
    \lim_{|x_i|\to\infty}(x_i+ch)e^{\left(\frac{\sigma^2-\tau^2}{\sigma^2\tau^2}\right)\left(x-\frac{h\sigma^2\tau^2}{\sigma^2-\tau^2}\right)^2}
    =0
\end{align}
and it is continuous, there exists $M>0$ so that
\begin{align}
    \left|\frac{2}{\sigma^2}(x_i+ch)e^{\left(\frac{\sigma^2-\tau^2}{\sigma^2\tau^2}\right)\left(x-\frac{h\sigma^2\tau^2}{\sigma^2-\tau^2}\right)^2}e^{\frac{h}{\sigma^2}\left(1+\frac{\tau^2}{\tau^2-\sigma^2}\right)}\right|
    \leq M.
\end{align}
Rearranging, we have that
\begin{equation}
    \left|\frac{2}{\sigma^2}(x_i+ch)e^{-\nicefrac{(x_i+ch)^2}{\sigma^2}}\right|
    \leq Me^{-\nicefrac{x_i^2}{\tau^2}}.
\end{equation}
So,
\begin{align}
    \left|f(\bmx+\sigma\bmu)(x_i+ch)e^{-\nicefrac{(x_i+ch)^2}{\sigma^2}}e^{-\sum_{j\not= i}\nicefrac{x_j^2}{\sigma^2}}\right|
    \leq M|f(\bmx+\sigma\bmu)|e^{-\nicefrac{\|\bmx\|^2}{\tau^2}}.
\end{align}
Since $Mf\in L^1(\mathbb{R}^d,k_{\tau})$ (by Lemma~\ref{lem:finl1}), we can use the dominated convergence theorem to show
\begin{align}
\begin{split}
    \lim_{h\to 0}\frac{f_{\sigma}(\bmx+he_i)-f_{\sigma}(\bmx)}{h}
    &=\int_{\mathbb{R}^d}\lim_{h\to 0}\left(f(\bmx+\sigma\bmu)(x_i+ch)e^{-\nicefrac{(x_i+ch)^2}{\sigma^2}}e^{-\sum_{j\not= i}\nicefrac{x_j^2}{\sigma^2}}\right)d\bmu\\
    &=\int_{\mathbb{R}^d}x_if(\bmx+\sigma\bmu)e^{-\nicefrac{\|x\|^2}{\sigma^2}}\;d\bmu\\
    &=\left(f\star\frac{\partial k_{\sigma}}{\partial x_i}\right)(\bmx).
\end{split}
\end{align}
Additionally, we have that
\begin{align}
\begin{split}
\label{eqn:gradient_in_integral_change_of_variables}
    &\lim_{h\to 0}\frac{f_{\sigma}(\bmx+he_i)-f_{\sigma}(\bmx)}{h}\\
    &\qquad=\int_{\mathbb{R}^d}\lim_{h\to 0}f(\bmx+\sigma\bmu)\left(\frac{e^{-\|u-\frac{h}{\sigma}e_i\|^2}-e^{-\|u\|^2}}{h}\right)\;d\bmu\\
    &\qquad=\frac{1}{h}\left(
        \int_{\mathbb{R}^d}\lim_{h\to 0}f(\bmx+\sigma\bmu)e^{-\|u-\frac{h}{\sigma}e_i\|^2}\;d\bmu
        -\int_{\mathbb{R}^d}\lim_{h\to 0}f(\bmx+\sigma\bmu)e^{-\|u\|^2}\;d\bmu
    \right)\\
    &\qquad=\frac{1}{h}\left(
        \int_{\mathbb{R}^d}\lim_{h\to 0}f(\bmx+he_i+\sigma\bmu)e^{-\|u\|^2}\;d\bmu
        -\int_{\mathbb{R}^d}\lim_{h\to 0}f(\bmx+\sigma\bmu)e^{-\|u\|^2}\;d\bmu
    \right)\\
    &\qquad=\int_{\mathbb{R}^d}\lim_{h\to 0}\frac{f(\bmx+he_i+\sigma\bmu)-f(\bmx+\sigma\bmu)}{h}e^{-\|u\|^2}\;d\bmu\\
    &\qquad=\left(\frac{\partial f}{\partial x_i}\star k_{\sigma}\right)(\bmx).
\end{split}
\end{align}
This shows that for and $i$
\begin{equation}
    \frac{\partial f_{\sigma}}{\partial x_i}(\bmx)
    =\left(\frac{\partial f}{\partial x_i}\star k_{\sigma}\right)(\bmx)
    =\left(f\star \frac{\partial k_{\sigma}}{\partial x_i}\right)(\bmx).
\end{equation}
\end{proof}

\begin{rem}
\label{rem:l-smooth_functions_satisfy_heat_eqn}
A very similar proof shows that
\begin{equation}
    \frac{\partial}{\partial \sigma}f_{\sigma}(\bmx)
    =\left(f\star\frac{\partial k_{\sigma}}{\partial\sigma}\right)(\bmx).
\end{equation}
This exact same proof justifies
\begin{equation}
    \frac{\partial^2}{\partial x_i^2}f_{\sigma}(\bmx)
    =\frac{\partial}{\partial x_i}\left(\frac{\partial f}{\partial x_i}\star k_{\sigma}\right)(\bmx)
    =\left(\frac{\partial^2 f}{\partial x_i^2}\star k_{\sigma}\right)(\bmx)
\end{equation}
for $L$-smooth functions, since $\frac{\partial}{\partial x_i}f$ is $L$-Lipschitz.
Then using the same change of variables in (\ref{eqn:gradient_in_integral_change_of_variables}) twice, we can switch the second derivative to the Gaussian.
This shows that for $L$-smooth functions, $f\star k_{\sigma}$ satisfies the heat equation (\ref{eqn:smoothing_heateqn}).
\end{rem}

\section{Proofs of main convergence results}
\label{app:convergence}

\subsection*{Proof of Theorem~\ref{thm:GSmoothGD_convex}}
\label{app:GSmoothGD_convex}
Suppose $f$ is convex and $L$-smooth.
Let $L^1_k=L$ for all $k\in\mathbb{N}$.
By Lemma \ref{lem:fsigmalsmoothandconvex}, $f_{\sigma_k}$ is $L^1_k$-smooth for each $k\in\mathbb{N}$.
As such, we have
\begin{align}
\begin{split}\label{eqn:quickproof2}
    f_{\sigma_{k+1}}(\bmx_{k+1})
    &\leq f_{\sigma_{k+1}}(\bmx_{k})+\langle\nabla f_{\sigma_{k+1}}(\bmx_{k}),\bmx_{k+1}-\bmx_{k}\rangle +\frac{1}{2}L^1_{k+1}\|\bmx_{k+1}-\bmx_{k}\|^2\\
    &=f_{\sigma_{k+1}}(\bmx_{k})-t\|\nabla f_{\sigma_{k+1}}(\bmx_{k})\|^2+\frac{1}{2}L^1_{k+1}t^2\|\nabla f_{\sigma_{k+1}}(\bmx_{k})\|^2\\
    &\leq f_{\sigma_{k+1}}(\bmx_{k})-\frac{1}{2}t\|\nabla f_{\sigma_{k+1}}(\bmx_{k})\|^2
\end{split}
\end{align}
We can combine (\ref{eqn:quickproof2}) with Lemma~\ref{lem:differentsmoothingvalues} to see that
\begin{equation}\label{eqn:quickproof4}
    f_{\sigma_k}(\bmx_{k})
    \leq f_{\sigma_k}(\bmx_{k-1})
    \leq f_{\sigma_{k-1}}(\bmx_{k-1})+L_k^2
\end{equation}
where $L_k^2=\frac{Ld}{4}\max(0,\sigma_k^2-\sigma_{k-1}^2)$ is the bound from Lemma~\ref{lem:differentsmoothingvalues}.

Using Lemma \ref{lem:fsigmalsmoothandconvex}, as $f$ is convex, so is $f_{\sigma_k}$ for every $k$. This means for any $\bmx$
\begin{equation}
    f_{\sigma_k}(\xstar)\geq f_{\sigma_k}(\bmx)+\langle\nabla f_{\sigma_k}(\bmx),\xstar-\bmx\rangle,
\end{equation}
which can be rewritten as
\begin{equation}\label{eqn:quickproof1}
    f_{\sigma_k}(\bmx)\leq f_{\sigma_k}(\xstar)+\langle\nabla f_{\sigma_k}(\bmx),\bmx-\xstar\rangle.
\end{equation}

Define $L^3_k$ as $\frac{Ld\sigma_k^2}{4}$, which is from Lemma~\ref{lem:differentsmoothingvalues}.
Then, combining the above, we have the following: 
\begin{align}
\begin{split}\label{eqn:quickproof3}
    0
    &\stackrel{\text{Lem \ref{lem:fsigmagreaterthanf}}}{\leq}
    f_{\sigma_{k+1}}(\bmx_{k+1})-f(\xstar)\\
    &=f_{\sigma_{k+1}}(\bmx_{k+1})-f_{\sigma_{k+1}}(\xstar)+f_{\sigma_{k+1}}(\xstar)-f(\xstar)\\
    &\stackrel{\text{Lem \ref{lem:differentsmoothingvalues}}}{\leq}
    f_{\sigma_{k+1}}(\bmx_{k+1})-f_{\sigma_{k+1}}(\xstar)+L_{k+1}^3\\
    &\stackrel{(\ref{eqn:quickproof2})}{\leq}
    f_{\sigma_{k+1}}(\bmx_{k})-\frac{1}{2}t\|\nabla f_{\sigma_{k+1}}(\bmx_{k})\|^2-f_{\sigma_{k+1}}(\xstar)+L_{k+1}^3\\
    &\stackrel{(\ref{eqn:quickproof1})}{\leq}
    \langle\nabla f_{\sigma_{k+1}}(\bmx_{k}),\bmx_{k}-\xstar\rangle-\frac{1}{2}t\|\nabla f_{\sigma_{k+1}}(\bmx_{k})\|^2+L_{k+1}^3.
\end{split}
\end{align}
Now, repeating the computation done in the proof of gradient descent, we have the following:
\begin{align}
\begin{split}\label{eqn:quickproof4again}
    f_{\sigma_{k+1}}(\bmx_{k+1})-f(\xstar)
    &\leq\langle\nabla f_{\sigma_{k+1}}(\bmx_{k}),\bmx_{k}-\xstar\rangle-\frac{1}{2}t\|\nabla f_{\sigma_{k+1}}(\bmx_{k})\|^2+L_{k+1}^3\\
    &=\frac{1}{2t}\left(\|\bmx_{k}-\xstar\|^2-\|\bmx_{k+1}-\xstar\|^2\right)+L_{k+1}^3
\end{split}
\end{align}

Summing over the steps:
\begin{align}
\begin{split}\label{eqn:quickproof5}
    \sum_{i=0}^{k-1}\left(f_{\sigma_{i+1}}(\bmx_{i+1})-f(\xstar)\right)
    &\stackrel{(\ref{eqn:quickproof4})}{\leq}\frac{1}{2t^*}\sum_{i=0}^{k-1}\left(\|\bmx_i-\xstar\|^2-\|\bmx_{i+1}-\xstar\|^2\right) + \sum_{i=0}^{k-1}L_{i+1}^3\\
    &\leq\frac{1}{2t}\|\bmx_0-\xstar\|^2 + \sum_{i=0}^{k-1}L_{i+1}^3\\
\end{split}
\end{align}
By (\ref{eqn:quickproof4}), for $k>i\geq 1$ we have
\begin{equation}
    f_{\sigma_k}(\bmx_{k})
    \leq f_{\sigma_{k-1}}(\bmx_{k-1}) + L_k^2
    \leq\cdots
    \leq f_{\sigma_{k-i}}(\bmx_{k-i}) + \sum_{j=1}^{i}L_{k-j+1}^2.
\end{equation}
This means we have (where $t_k^*=t$)
\begin{align}
\begin{split}\label{eqn:quickproof6}
    k\left(f_{\sigma_{k}}(\bmx_{k})-f(\xstar)\right)
    &=(f_{\sigma_{k}}(\bmx_{k})-f(\xstar))+\sum_{i=1}^{k-1}(f_{\sigma_{k}}(\bmx_{k})-f(\xstar))\\
    &=\sum_{i=1}^{k}(f_{\sigma_{i}}(\bmx_i)-f(\xstar))+\sum_{i=1}^{k-1}iL_{i+1}^2\\
    &\leq\frac{1}{2t_k^*}\|\bmx_0-\xstar\|^2 + \sum_{i=0}^{k-1}L_{i+1}^3+\frac{Ld}{4}\sum_{i=1}^{k-1}iL_{i+1}^2.
\end{split}
\end{align}

Finally, applying Lemma~\ref{lem:fsigmagreaterthanf} and dividing by $k$,
\begin{equation}\label{eqn:quickproof7}
    f(\bmx_{k})-f(\xstar)\leq
    f_{\sigma_{k}}(\bmx_{k})-f(\xstar)\leq
    \frac{1}{2t_k^*k}\|\bmx_0-\xstar\|^2+\frac{1}{k}\left(
    \sum_{i=1}^{k}L_i^3
    +\sum_{i=2}^{k}iL_i^2
    \right),
\end{equation}
which completes the proof.

\begin{rem}[$f$ is convex and $M$-Lipschitz]
The above proof holds for $M$-Lipschitz instead of $L$-smooth with the following modifications:
\begin{itemize}
    \item[(i)] 
    replace $t$ with $t_{k+1}$ in equations (\ref{eqn:quickproof2}), (\ref{eqn:quickproof3}), and (\ref{eqn:quickproof4again});
    \item[(ii)] 
    define $t_k^*=\min\{t_i|i\in\{1,...,k\}\}$ and replace $t$ with $t_k^*$ in equations (\ref{eqn:quickproof5}), (\ref{eqn:quickproof6}), and (\ref{eqn:quickproof7});
    \item[(iii)] 
    Define $L^1_k=\frac{M\sqrt{2d}}{\sigma_k}$, 
    $L^2_k=M\sqrt{(\frac{d}{2})\max(0,\sigma_k^2-\sigma_{k-1}^2)}$, and 
    $L^3_k=M\sigma_k\sqrt{\frac{d}{2}}$; and 
    \item[(iv)] 
    if $f$ is not differentiable at $\bmx_{k}$, then we require $f_{\sigma_{k}}(\bmx_{k})-f(\xstar)$ on the left-hand-side of \eqref{eq:f1}.
\end{itemize}
\end{rem}
%

\subsection*{Proof of Theorem~\ref{thm:GSmoothGD_non-convex}}
\label{app:GSmoothGD_non-convex}
From the proof of Theorem~\ref{thm:GSmoothGD_convex}, since $f$ is $L$-smooth (regardless of convexity)
\begin{equation}
    f_{\sigma_{k+1}}(\bmx_{k+1})
    \leq f_{\sigma_{k+1}}(\bmx_{k})-\frac{t}{2}\|\nabla f_{\sigma_{k+1}}(\bmx_{k})\|^2.
\end{equation}
As expected, this means $f_{\sigma_{k+1}}(\bmx_{k})\geq f_{\sigma_{k+1}}(\bmx_{k+1})$.

For $k\geq 1$,
\begin{align}
\begin{split}\label{eqn:ncgdwss1}
    \|\nabla f_{\sigma_{k+1}}(\bmx_{k})\|^2
    &\leq\frac{2}{t}\Big(f_{\sigma_{k+1}}(\bmx_{k})-f_{\sigma_{k+1}}(\bmx_{k+1})\Big)\\
    &=
    \frac{2}{t}\Big(f_{\sigma_{k}}(\bmx_{k})-f_{\sigma_{k+1}}(\bmx_{k+1})\Big) + \frac{Ld}{2t}|\sigma_{k+1}^2-\sigma_k^2|.
\end{split}
\end{align}

This means that
\begin{align}
\begin{split}
    &k\min_{i=1,...,k}\|\nabla f(\bmx_i)\|^2\\
    &\qquad\leq\sum_{i=1}^{k}\|\nabla f(\bmx_i)\|^2\\
    &\qquad\leq2\sum_{i=1}^{k}\|\nabla f_{\sigma_{i+1}}(\bmx_i)\|^2+\frac{L^2(6+d)^3}{4}\sum_{i=1}^{k}\sigma_{i+1}^2\\
    &\qquad\leq
        \frac{4}{t}\sum_{i=1}^{k}\big(f_{\sigma_i}(\bmx_i)-f_{\sigma_{i+1}}(\bmx_{i+1})\big)
        +\frac{Ld}{2t}\sum_{i=1}^{k}|\sigma_{i+1}^2-\sigma_i^2|
        +\frac{L^2(6+d)^3}{4}\sum_{i=1}^{k}\sigma_{i+1}^2
\end{split}
\end{align}
using Lemma 4 from \cite{nesterov2017random} on the second inequality.
Dividing both sides by $k$ gives the result.

\subsection*{Proof of Proposition \ref{prop:complexity}}
\label{app:complexity}
Let $f$ be $L$-smooth, so that $f_{\sigma}$ is smooth for any $\sigma\geq 0$.
We assume nothing about the convexity of $f$.
From the proof of Theorem~\ref{thm:GSmoothGD_non-convex}, we have that
\begin{align}
\label{eqn:proofofpropcomplexity}
    \frac{1}{k}\sum_{i=1}^{k}\|\nabla f(\bmx_i)\|^2
    &\leq\frac{1}{k}\left[
        \frac{4}{t}\big(f(\bmx_0)-f(\xstar)\big)
        +\left(\frac{Ld}{t}+\frac{L^2(6+d)^3}{4}\right)\sum_{i=1}^{k+1}\sigma_{i}^2
    \right]\\
    &=O\left(\frac{1+d^3\sum_{i=1}^{k+1}\sigma_i^2}{k}\right).
\end{align}
This means for at least some $j$, $\|\nabla f(x_j)\|^2$ is at most the LHS of (\ref{eqn:proofofpropcomplexity}).
Hence, if $(\sigma_i)_{i=1}^{\infty}\in l^2(\mathbb{R}^+)$ and we want $\|\nabla f(x_j)\|<\epsilon$, we need at least
\begin{equation}
    \frac{1}{\epsilon^2}\left[
        \frac{4}{t}\big(f(\bmx_0)-f(\xstar)\big)
        +\left(\frac{Ld}{t}+\frac{L^2(6+d)^3}{4}\right)\sum_{i=1}^{k+1}\sigma_{i}^2
    \right]
    =O\left(\frac{1+d^3}{\epsilon^2}\right)
\end{equation}
iterations, which completes the proof.

\subsection*{Proof of Lemma~\ref{lem:varianceofmcGSmoothGD}}
\label{app:varianceofmcGSmoothGD}
First, we have
\begin{equation}
    \text{Var}(\bmg_{\sigma}(\bmx;N))
    =\frac{1}{N^2}\text{Var}\left(\sum_{n=1}^{N}\delta_{\sigma}(\bmx;\bmu_n)\bmu_n\right)
    =\frac{1}{N}\text{Var}(\bmg_{\sigma}(\bmx;1)).
\end{equation}
Second, since $E(\bmg_{\sigma}(\bmx;1))=\nabla f_{\sigma}(\bmx)$,
\begin{align}
\begin{split}
    \text{Var}(\bmg_{\sigma}(\bmx;1))
    &=E\Big(\big(\bmg_{\sigma}(\bmx;1)-\nabla f_{\sigma}(\bmx))\big)\big(\bmg_{\sigma}(\bmx;1)-\nabla f_{\sigma}(\bmx))\big)^T\Big)\\
    &=E\big(\bmg_{\sigma}(\bmx;1)\bmg_{\sigma}(\bmx;1)^T\big)-\nabla f_{\sigma}(\bmx)\nabla f_{\sigma}(\bmx)^T\\
    &=E\big(\delta_{\sigma}(\bmx;\bmu_1)^2\bmu_1\bmu_1^T\big)-\nabla f_{\sigma}(\bmx)\nabla f_{\sigma}(\bmx)^T.
\end{split}
\end{align}
So,
\begin{equation}
    \text{tr}(\text{Var}(\bmg_{\sigma}(\bmx;1)))
    =E\big(\delta_{\sigma}(\bmx;\bmu_1)^2\|\bmu_1\|^2\big)-\|\nabla f_{\sigma}(\bmx)\|^2
\end{equation}

Finally,
\begin{align}
\begin{split}
    E(\|\bmg_{\sigma}(\bmx;N)\|^2)
    &=\text{tr}(\text{Var}(\bmg_{\sigma}(\bmx;N)))+\|\nabla f_{\sigma}(\bmx)\|^2\\
    &=\frac{1}{N}E\big(\delta_{\sigma}(\bmx;\bmu_1)^2\|\bmu_1\|^2\big)+\left(1-\frac{1}{N}\right)\|\nabla f_{\sigma}(\bmx)\|^2.
\end{split}
\end{align}

\subsection*{Proof of Theorem~\ref{thm:mcGSmoothGD}}
\label{app:convergenceofmcGSmoothGD}
This proof is a modification of the proof of Theorem~\ref{thm:GSmoothGD_non-convex}.

First, if $\bmx\in\mathbb{R}^d$, $\sigma>0$, and $\bmx'=\bmx-t\bmg_{\sigma}(\bmx;N)$, then since $f$ is $L$-smooth
\begin{align}
\begin{split}
    f_{\sigma}(\bmx')
    &\leq f_{\sigma}(\bmx)+\langle\nabla f_{\sigma}(\bmx),\bmx'-\bmx\rangle +\frac{L}{2}\|\bmx'-\bmx\|^2\\
    &\leq f_{\sigma}(\bmx)-t\langle\nabla f_{\sigma}(\bmx),\bmg_{\sigma}(\bmx;N)\rangle +\frac{Lt^2}{2}\|\bmg_{\sigma}(\bmx;N)\|^2.
\end{split}
\end{align}
The expected value of the inner product in the second term is given by
\begin{equation}
    E(\langle\nabla f_{\sigma}(\bmx),\bmg_{\sigma}(\bmx;N)\rangle)
    =\langle\nabla f_{\sigma}(\bmx),E(\bmg_{\sigma}(\bmx;N))\rangle
    =\|\nabla f_{\sigma}(\bmx)\|^2.
\end{equation}
This means
\begin{align}
\begin{split}
    E(f_{\sigma}(\bmx'))
    &\leq f_{\sigma}(\bmx)-tE(\langle\nabla f_{\sigma}(\bmx),\bmg_{\sigma}(\bmx;N)\rangle) +\frac{Lt^2}{2}E(\|\bmg_{\sigma}(\bmx;N)\|^2)\\
    &\leq f_{\sigma}(\bmx)-t\|\nabla f_{\sigma}(\bmx)\|^2 +\frac{Lt^2}{2}E(\|\bmg_{\sigma}(\bmx;N)\|^2).
\end{split}
\end{align}
Applying the Lemma~\ref{lem:varianceofmcGSmoothGD} shows
\begin{align}
\begin{split}
    E(f_{\sigma}(\bmx'))
    &\leq f_{\sigma}(\bmx)-t\|\nabla f_{\sigma}(\bmx)\|^2 +\frac{Lt^2}{2}E(\|\bmg_{\sigma}(\bmx;N)\|^2)\\
    &\leq f_{\sigma}(\bmx)-\frac{t}{2}\|\nabla f_{\sigma}(\bmx)\|^2 +\frac{Lt^2}{2N}E(\delta_{\sigma}(\bmx;\bmu)^2\|\bmu\|^2).
\end{split}
\end{align}
Rearranging gives
\begin{align}
\begin{split}
    \|\nabla f_{\sigma}(\bmx)\|^2
    &\leq\frac{2}{t}E\left(f_{\sigma}(\bmx)-f_{\sigma}(\bmx')\right)+\frac{Lt}{N}E(\delta_{\sigma}(\bmx;\bmu)^2\|\bmu\|^2).
\end{split}
\end{align}

Applying this to our gradient descent sequence, we have
\begin{align}
\begin{split}
    E(\|\nabla f_{\sigma_{k+1}}(\bmx_k)\|^2)
    &\leq \frac{2}{t}E(f_{\sigma_{k+1}}(\bmx_k)-f_{\sigma_{k+1}}(\bmx_{k+1}))+\frac{Lt}{N}E(\delta_{\sigma_{k+1}}(\bmx_k;\bmu)^2\|\bmu\|^2).
\end{split}
\end{align}
Applying Lemma~\ref{lem:differentsmoothingvalues}
\begin{multline}
    E(\|\nabla f_{\sigma_{k+1}}(\bmx_k)\|^2)
    \leq\frac{2}{t}E(f_{\sigma_{k}}(\bmx_k)-f_{\sigma_{k+1}}(\bmx_{k+1}))\\
    +\frac{Lt}{N}E(\delta_{\sigma_{k+1}}(\bmx_k;\bmu)^2\|\bmu\|^2)+\frac{Ld}{2t}|\sigma_{k+1}^2-\sigma_k^2|.
\end{multline}

From \cite[Lemma 4]{nesterov2017random}, we have that
\begin{equation}
    \|\nabla f(\bmx)\|^2
    \leq 2\|\nabla f_{\sigma}(\bmx)\|^2 +\frac{\sigma^2L^2(d+3)^3}{16}.
\end{equation}
Combining this with our computation gives
\begin{multline}
    E(\|\nabla f(\bmx_k)\|^2)
    \leq\frac{4}{t}E(f_{\sigma_{k}}(\bmx_k)-f_{\sigma_{k+1}}(\bmx_{k+1}))+\frac{2Lt}{N}E(\delta_{\sigma_{k+1}}(\bmx_k;\bmu)^2\|\bmu\|^2)\\+\frac{Ld}{t}|\sigma_{k+1}^2-\sigma_k^2|+\frac{\sigma_{k+1}^2L^3t(d+3)^3}{16N}.
\end{multline}

From Theorem 4 of \cite{nesterov2017random}, we have that for any $\bmx$
\begin{equation}
    E(\delta_{\sigma}(\bmx;\bmu)^2\|\bmu\|^2)
    \leq\frac{L^2\sigma^2(d+6)^3}{16}+(d+4)\|\nabla f(\bmx)\|^2.
\end{equation}
This means
\begin{multline}
    E(\|\nabla f(\bmx_k)\|^2)
    \leq\frac{4}{t}E(f_{\sigma_{k}}(\bmx_k)-f_{\sigma_{k+1}}(\bmx_{k+1}))+\frac{L^3\sigma_{k+1}^2t(d+6)^3}{8N}\\+\frac{2Lt(d+4)}{N}E(\|\nabla f(\bmx_k)\|^2)+\frac{Ld}{t}|\sigma_{k+1}^2-\sigma_k^2|+\frac{\sigma_{k+1}^2L^3t(d+3)^3}{16N}.
\end{multline}
So, with $\sigma_0=0$,
\begin{multline}
    k\left(1-\frac{2Lt(d+4)}{N}\right)\min_{i=0,...,k-1}E(\|\nabla f(\bmx_i)\|^2)
    \leq\frac{4}{t}(f(\bmx_0)-f_{\sigma_{k}}(\bmx_{k}))\\
    +\left(\frac{L^3t(d+6)^3}{8N}+\frac{L^3t(d+3)^3}{16N}\right)\sum_{i=1}^{k}\sigma_{i}^2
    +\frac{Ld}{t}\sigma_1^2+\frac{Ld}{t}\sum_{i=2}^{k}|\sigma_{i}^2-\sigma_{i-1}^2|.
\end{multline}

If $\frac{1}{2L(d+4)}>t$ and $N\geq 1$, then $1-\frac{2Lt(d+4)}{N}>0$.
In which case, after dividing by $k$,
\begin{multline}
    \min_{i=0,...,k-1}E(\|\nabla f(\bmx_i)\|^2)
    \leq\frac{1}{k}\left(\frac{N}{N-2Lt(d+4)}\right)\Bigg(\frac{4(f(\bmx_0)-f_{\sigma_{k}}(\bmx_{k}))}{t}\\
    +\left(\frac{L^3t(d+6)^3}{4N}+\frac{2Ld}{t}\right)\sum_{i=1}^{k}\sigma_{i}^2\Bigg)
\end{multline}


\bibliographystyle{plain}
\bibliography{GSGD}

\end{document}